\documentclass[a4paper,11pt]{amsart}
\usepackage{graphicx}
\usepackage{amssymb}
\usepackage[french, english]{babel}
\usepackage[utf8]{inputenc}
\usepackage{amsmath,amsfonts,amssymb,amsthm}
\usepackage[T1]{fontenc}
\usepackage{fullpage}
\usepackage{hyperref,enumitem}
\usepackage{mathrsfs}
\usepackage{relsize}
\usepackage{tikz,pgfplots,todonotes}
\usepackage{bm,bbm,cleveref,csquotes}
\usepackage{appendix,comment}
\usepackage{pifont}
\usepackage{soul}
\usetikzlibrary{patterns}

\usepackage{pgfplots}
\definecolor{darkgreen}{rgb}{0.0, 0.5, 0.0}

\newlength{\bibitemsep}
\newlength{\bibparskip}
\let\oldthebibliography\thebibliography
\renewcommand\thebibliography[1]{%
  \oldthebibliography{#1}%
  \setlength{\parskip}{\bibitemsep}%
  \setlength{\itemsep}{\bibparskip}%
}

\newcommand{\eps}{\varepsilon}

\newcommand{\dd}{\textup{d}}
\newcommand{\ee}{\textup{e}}

\newcommand{\R}{\mathbb R}

\newcommand{\cO}{\mathcal O}

\newcommand{\dens}[2]{\mathrm{dens}\left(#1,#2\right)}
\newcommand{\card}[1]{\left|\left\{ #1 \right\}\right|}

\newcommand{\One}{\bm 1}
\DeclareMathOperator{\Leb}{Leb}
\DeclareMathOperator{\dTV}{d_{TV}}
\DeclareMathOperator{\Av}{Av}
\DeclareMathOperator{\inv}{inv}
\DeclareMathOperator{\occ}{occ}
\DeclareMathOperator{\noninv}{noninv}
\DeclareMathOperator{\proj}{proj}
\DeclareMathOperator{\Dyck}{Dyck}

\newcommand{\PAvTDU}{\mathcal P_{\Av(321)}}
\newcommand{\PAvDTU}{\mathcal P_{\Av(231)}}
\newtheorem{thm}{Theorem}
\newtheorem{prop}[thm]{Proposition}
\newtheorem{cor}[thm]{Corollary}
\newtheorem{lem}[thm]{Lemma}

\theoremstyle{definition}

\newtheorem{rem}[thm]{Remark}

\newcommand{\lrtriangle}{{\begin{tikzpicture}[scale=.4]
  \draw (0,0) -- (0.5,0) -- (0.5,0.5);
  \draw[thick,dotted] (0,0) -- (0.5,0.5);
\end{tikzpicture}}}
\newcommand{\ultriangle}{{\, \begin{tikzpicture}[scale=.4]
  \draw (0,0) -- (0,0.5) -- (0.5,0.5);
  \draw[thick,dotted] (0,0) -- (0.5,0.5);
\end{tikzpicture}}}

\title[Pattern-avoiding Mallows permutations]
{Large deviation principles\\ for pattern-avoiding permutations,\\
and limit shapes for constrained Mallows permutations}
\author{Thomas Budzinski}
\address{ENS de Lyon, CNRS, UMPA}
\email{thomas.budzinski@ens-lyon.fr}
\author{Victor Dubach}
\address{Department of Mathematics, Uppsala University, Sweden}
\email{victor.dubach@math.uu.se}
\author{Valentin Féray}
\address{Université de Lorraine, CNRS, IECL, F-54000 Nancy, France}
\email{valentin.feray@univ-lorraine.fr}
\author{Mohamed Slim Kammoun}
\address{Université de Poitiers, LMA, France}
\email{skammoun@math.univ-poitiers.fr}
\author{Mylène Maïda}
\address{Universit\'e de Lille, CNRS, Inria, UMR 8524 - Laboratoire Paul Painlev\'e, F-59000 Lille, France}
\email{mylene.maida@univ-lille.fr}
\date{}



\begin{document}

\maketitle
\begin{abstract}
    We study Mallows random permutations conditioned to avoid a given pattern $\alpha$ of length~$3$.
    When the bias parameter is of the form $\ee^{\beta/n}$, we prove that these permutations converge to a non-trivial explicit deterministic permuton that depends on the pattern $\alpha$ and
    on the parameter $\beta$.
    Along the way, we provide parametrizations for $\alpha$-avoiding permutons, and establish a large deviation principle for uniform $\alpha$-avoiding permutations.
	As a byproduct of the proof, we also obtain asymptotic estimates of two versions of $q$-Catalan numbers
	in the regime $q=\ee^{\beta/n}$.
\end{abstract}

\begin{figure}[h]
    \centering
    \includegraphics[scale=.41]{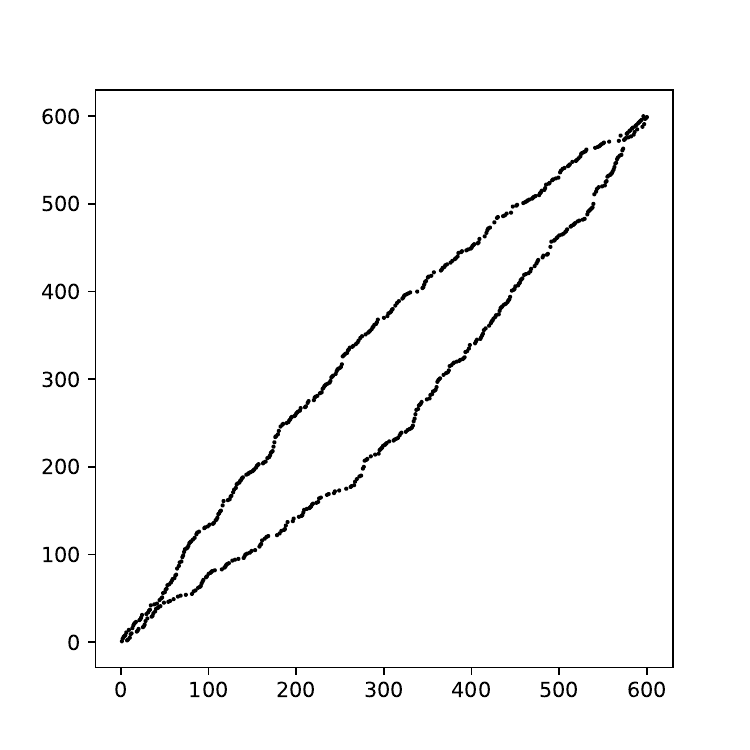}
    \includegraphics[scale=.41]{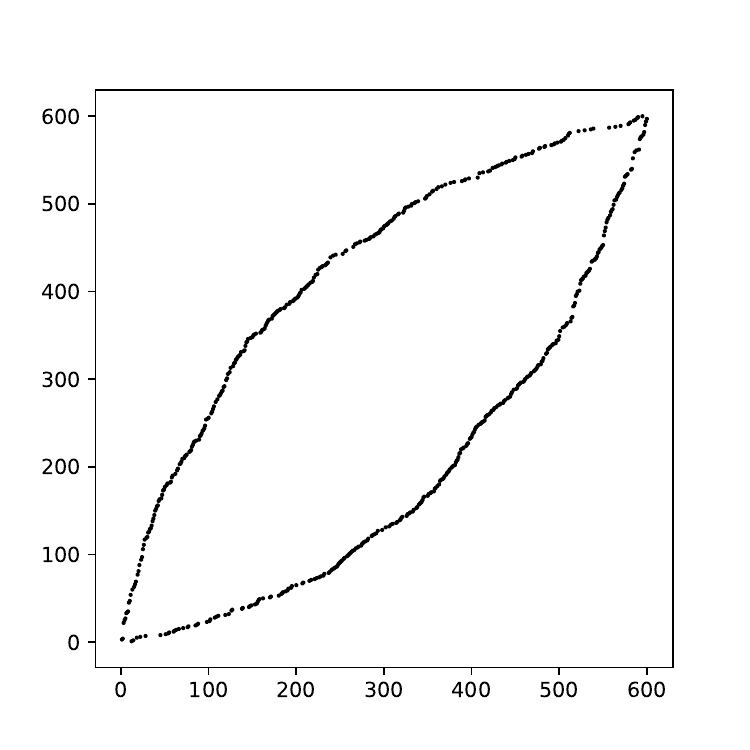}
    \includegraphics[scale=.41]{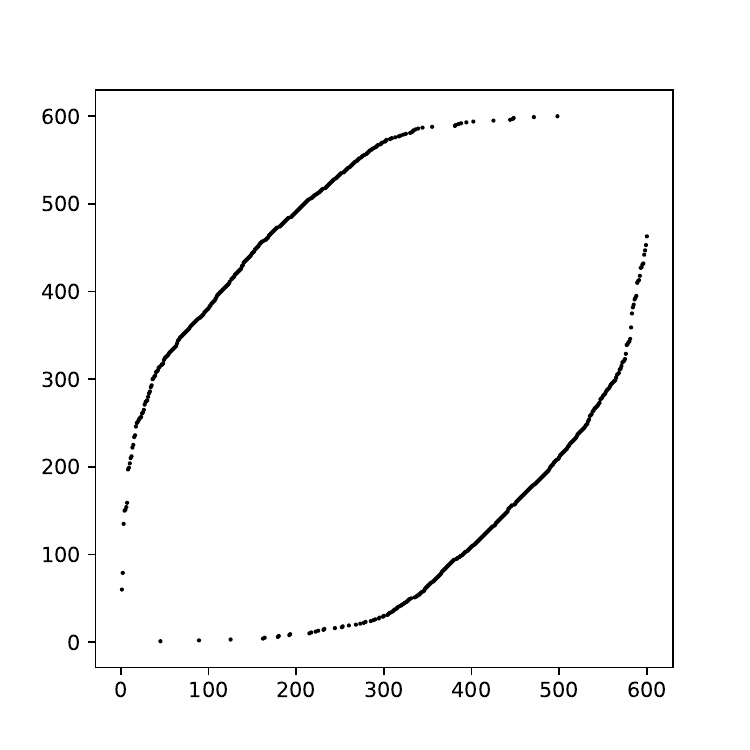} \\
    \includegraphics[scale=.41]{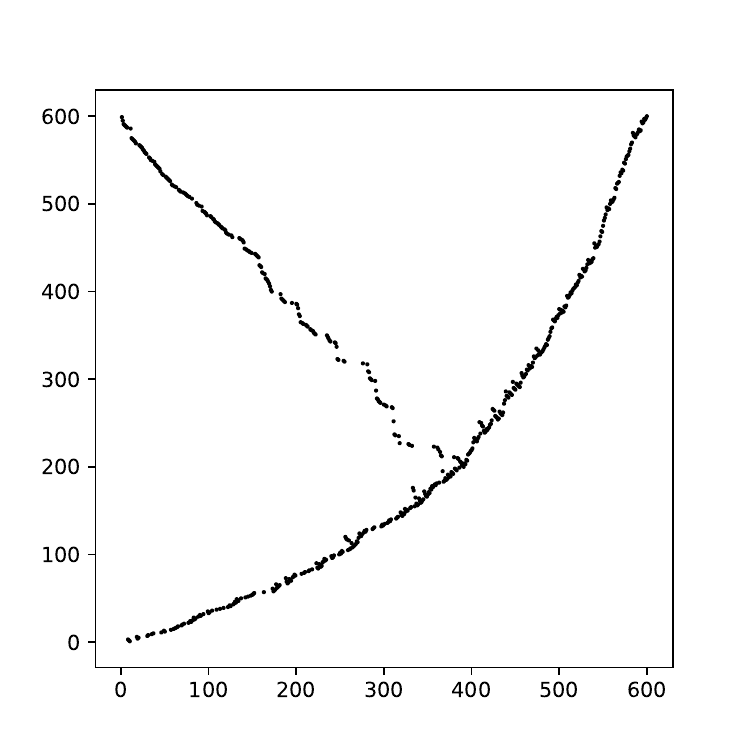}
    \includegraphics[scale=.41]{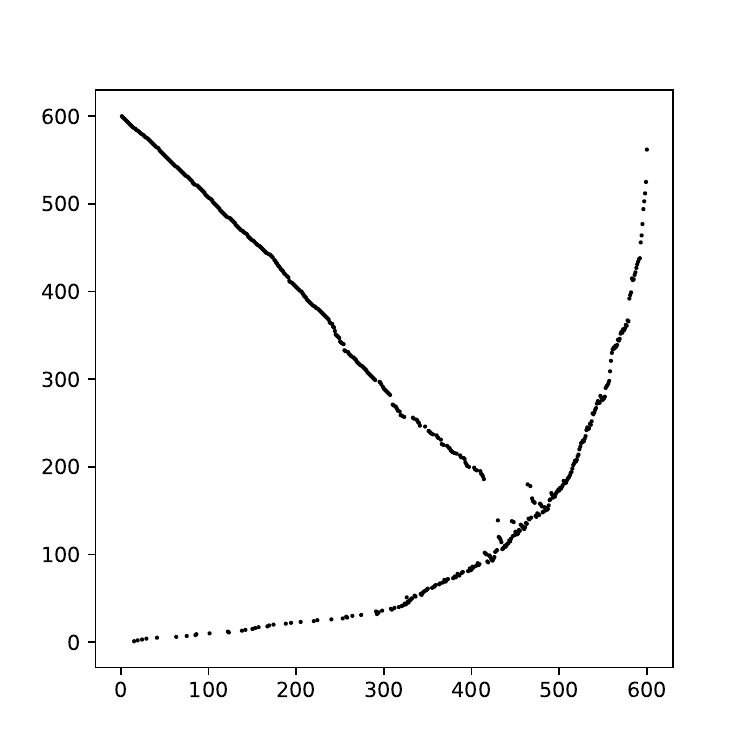}
    \includegraphics[scale=.41]{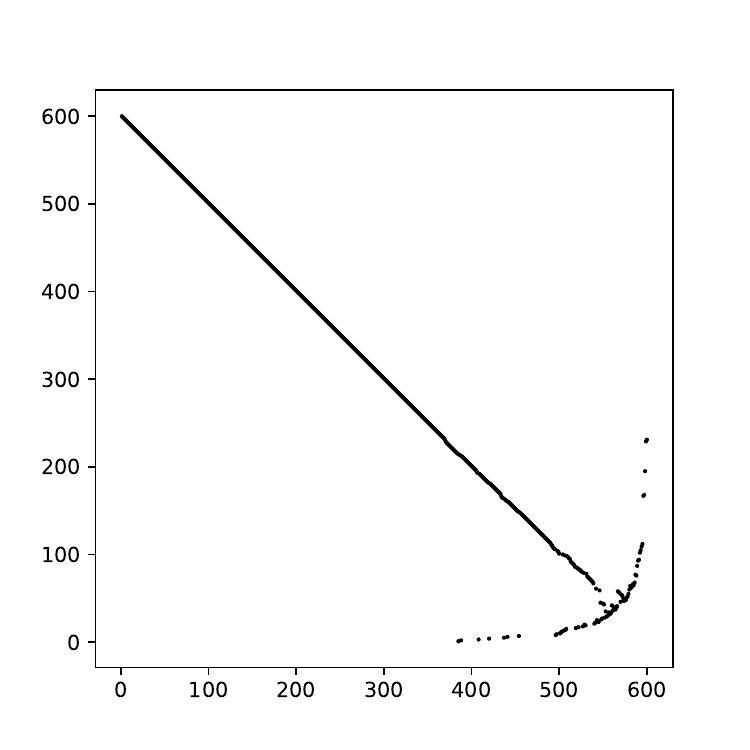}
    \caption{Top line: 
    simulations of $321$-avoiding Mallows permutations with size $600$ and bias parameter $q=\ee^{\beta/n}$, for $\beta\in\{1,3,12\}$. 
    Bottom line: 
    likewise with $231$-avoiding Mallows permutations. 
    Here, and throughout the paper, a permutation $\sigma$ is represented graphically by the set $\{(i,\sigma(i)), 1 \le i \le n\}$.}
    \label{fig:simulations_Constraint_Mallows}
\end{figure}
\newpage

\section{Introduction}

\subsection{Overview}
\label{ssec:background}
{\em Mallows/pattern-avoiding random permutations.}
Random permutations are among the most standard objects studied in combinatorial theory,
with notable applications in statistics, analysis of algorithms, and statistical physics.
While earlier works have focused on uniform random permutations of a given size,
some non-uniform models have become popular, in particular the so-called Mallows model,
introducing a bias depending on the number of inversions of the permutations.
This model has been introduced by Mallows in 1957 \cite{mallows1957} in the context of statistical rankings,
and has been studied under many aspects in the last twenty years, see e.g.~\cite{dubach2024mallows} for references.
In another direction, a number of recent research works analyze uniform random permutations
subject to some constraints. 
The most common constraints consist in fixing the cycle type 
(see e.g. \cite{feray2024clt_conjugacy,dubach2025geometric,hamaker2025moments}),
or in avoiding some patterns. 
The latter line of research originates from articles of Madras, Pak and co-authors 
and has grown significantly in recent years; 
see e.g. \cite{madras2010random_pattern_avoiding,MinerPak,hoffman2019scaling,bassino2022scaling,borga2022strong-semi-baxter}.
\bigskip

Our first motivation for this work is to combine both points of view,
i.e., to study Mallows random permutations constrained to avoid some patterns.
In this paper, we focus on avoiding a fixed pattern of length 3,
and refer to this as \emph{constrained Mallows permutations}. 
We focus on $321$ (resp.~$231$)-avoiding permutations, since the study of $\alpha$-avoiding permutations for other patterns $\alpha$ of size 3 can be reduced to one of these two cases by symmetry. 
We recall that a permutation $\sigma$ avoids $321$ if it does not contain any decreasing subsequence of length $3$, and it avoids $231$ if there are no indices $i<j<k$ such that $\sigma(k) < \sigma(i) < \sigma(j)$. 
Also, we denote by $\inv(\sigma)$ the number of inversions of $\sigma$. 
For $\beta \in \R$, $\alpha$ in $\{231,321\}$ and $n \ge 1$, we let $\tau^{\beta,\alpha}_{n}$ be a random permutation of size $n$ such that, for any $\alpha$-avoiding permutation $\sigma$:
\begin{equation}\label{eq:def_Constraint_Mallows}
 \mathbb P(\tau^{\beta,\alpha}_{n} = \sigma) = \frac{\ee^{\frac{\beta}{n}\inv(\sigma)}}{Z_n^{\beta,\alpha}},
\end{equation}
where $Z_n^{\beta,\alpha} := \sum_{\sigma} \ee^{\frac{\beta}{n}\inv(\sigma)}$ is the appropriate normalization factor
(the sum is taken over all $\alpha$-avoiding permutation $\sigma$).
This is the Mallows distribution with parameter $q_n=\ee^{\beta/n}$ conditioned to avoid $\alpha$. 
The reason why we chose this scaling for $q_n$ will be explained after the statements of the main results.
\Cref{fig:simulations_Constraint_Mallows} shows samples of $\tau^{\beta,231}_{n}$
and $\tau^{\beta,321}_{n}$ which were obtained by simulating an appropriate Markov chain,
as explained in Appendix~\ref{sec: simulations by reversible MC}.
Upon looking at these simulations, one might guess that a non-trivial permuton limit exists for this model;
see \Cref{sec: permuton limits} for precise statements to that regard.
\medskip

As far as we are aware of, this question has not yet been investigated, although related models and questions
have appeared in the literature.
First, let us mention the work \cite{gorin2024six-vertex} of Gorin and Kenyon, motivated by the six vertex model
of statistical physics, who analyzed random Mallows permutations with a different type of constraints,
namely that the diagram of the permutation avoids some designated regions.
In another direction, Chelikavada and Pranzo \cite{chelikavada2023fixed_point_bias}, and more recently Park and Rizzolo \cite{park2025limittheoremsfixedpoint},
  studied pattern-avoiding permutations biased by their number of fixed points.
Lastly, in \cite{pinsky2021avoidance_Mallows}, Pinsky gives a number of results regarding the probability
that a Mallows random permutations avoids a pattern of size 3, 
but he does not study the resulting conditional model as we do here.
\bigskip

{\em Permutons, and large deviation principles.}
Among many aspects of random Mallows permutations avoiding some patterns,
we focus here on finding the scaling limit of such objects, in the sense of permutons
(see \Cref{ssec:permutons} for background on permutons).
A good tool for this are large deviation principles (LDPs), which are useful when looking for
the limit shapes of random objects distributed
according to biased distributions, such as random Mallows permutations.
It indeed transforms the probabilistic problem into some optimization problem on a space of analytic objects (in our case, permutons). See \cite{starr2009mallows,mukherjee2016exponential_permutations,kenyon2020fixed_densities} for applications of this principle to some random permutation models,
in particular to (unconditioned) Mallows random permutations.
The space of permutons, which are probability measures on $[0,1]^2$,
 is known to be a natural framework to state
an LDP for uniform random permutations \cite{trashorras2008large_deviations,kenyon2020fixed_densities}. 
However, this ``classical'' LDP cannot be used to study models of pattern-avoiding permutations, since the rate function is equal to $\infty$ on all permutons avoiding a given pattern\footnote{
This should not come as a surprise, since the number of permutations avoiding a given pattern behaves as $O(C^n)/n!$ by the Stanley--Wilf--Marcus--Tardös theorem (see, e.g., \cite[Chapter 4]{BonaLivre}).
}.
\bigskip

{\em Our contributions.} 
We overcome this problem by deriving specific LDPs for uniform $321$ (resp.~$231$)-avoiding permutations.
A first step in this direction is to find appropriate parametrizations of $321$ (resp.~$231$)-avoiding permutons.
Once we have found these parametrizations and established the associated LDPs,
we can translate the problem of finding the scaling limit of $321$ (resp.~$231$)-avoiding
Mallows permutations to an analytic optimization. 
It turns out that, in both cases,
this optimization problem can be solved explicitly,
yielding an explicit description of the limit shapes.
In addition to the limit shape results, solving these optimization problems
also yields asymptotic estimates for the partition functions $Z_n^{\beta,\alpha}$ (both for $\alpha=321$ and $231$).
This complements previous results of Pinsky
in the regime where the bias parameter $q$ is constant~\cite{pinsky2021avoidance_Mallows}
(recall that for us, $q=\exp(\beta/n)$).

In the remaining part of the introduction, we provide more details on these three steps 
(parametrization of pattern-avoiding permutations, LDPs, and limit shapes for constrained Mallows permutations),
which we believe are all interesting in their own right.
\medskip

\begin{rem}
The normalization constant $Z_n^{\beta,\alpha}$ appearing in \eqref{eq:def_Constraint_Mallows} is the generating polynomial of inversions for 231, resp.~321, avoiding permutations.
In the 231-avoiding case, the classical decomposition of 231-avoiding permutations around their maxima yields the following recurrence
\[Z^{\beta,231}_{n+1}= \sum_{i=0}^n q^{i} Z^{\beta,231}_i Z^{\beta,231}_{n-i},\]
where $q=\ee^{\beta/n}$ as usual. This is the same recurrence as the $q$-Catalan numbers defined as 
the area generating function for Dyck paths; see, e.g., \cite[Chapter 1]{haglund2008qt_catalan} (see also Proposition~\ref{prop: inv of 231 avoiding permutation} for a direct connection between inversions in 231-avoiding permutations
and Dyck path area).
Moving on now to $\alpha=321$, the quantity $Z^{\beta,321}_{n}$ also satisfies some recurrence equation, 
though this is harder to prove, see~\cite{cheng2013inversions321avoiding}.
We shall not need these recurrences in this paper.
\end{rem}\medskip

\subsection{First main results: parametrizations of 231 or 321-avoiding permutons}
\label{ssec:parametrization}
We recall that a permuton is a probability measure on the square $[0,1]^2$ with both marginals equal to the Lebesgue measure on $[0,1]$. 
We refer to Section~\ref{ssec:permutons} below for related precise definitions.
While there are many works studying the structure of the set of $\alpha$-avoiding permutations (see, e.g., \cite{vatter2015survey}),
it seems that the question of describing $\alpha$-avoiding permutons has not been looked up so far.
We open the way here by giving natural parametrizations of $231$ (resp.~$321$)-avoiding permutons.
Given a pattern $\alpha$ and a permuton $\mu$, we say that $\mu$ is $\alpha$-avoiding
if the density of $\alpha$ in $\mu$ is equal to $0$. Equivalently, $\mu$ is $\alpha$-avoiding
if and only if it is the limit of a sequence of $\alpha$-avoiding permutations.

Before stating our results, let us mention some related literature.
In the realm of graphs, a series of papers find natural parametrizations of graphons
that are limits of graphs in some specific classes, namely threshold graphs~\cite{diaconis2008threshold_graphons}, 
interval graphs~\cite{diaconis2013interval_graphons}, and string graphs~\cite{janson2017string_graphons}.
In some sense, we are looking at the analogue problems for permutations.
In another direction, in a recent paper~\cite{garbe2024pattern_avoiding_permutons},
Garbe, Hladk{\'y}, Kun and Pek{\'a}rkov{\'a} have looked at some general properties of pattern-avoiding permutations,
but were not interested in a complete parametrization of them as we are. 
Finally, a number of papers consider permuton limits of uniform random permutations avoiding some fixed patterns 
(see, e.g., \cite{bassino2022scaling,borga2022strong-semi-baxter,hohmeier2025permuton_limits} and references in these papers). 
While these limits necessarily are pattern-avoiding permutons, 
this is of course a different question than trying to describe the set of {\em all} permutons avoiding given patterns.
\medskip

We first state our parametrization for $231$-avoiding permutons,
as it may seem more natural to readers familiar with pattern-avoiding permutations.
Indeed, it is standard that $231$-avoiding permutations are entirely determined
by the positions and values of their right-to-left minima
(abbreviated ``RLM'' from here on out); see, e.g., \cite[Chapter 4]{BonaLivre}.
Let us define the RLM curve of a permuton $\mu$ by
\[
    f_\mu: x\in[0,1] \mapsto \max\{ y\in[0,1] : \mu([x,1]\times[0,y])=0 \} \,.
\]
Necessarily, $f_\mu$ is a càdlàg nondecreasing function with $f_\mu(x) \le x$ and $f_\mu(1)=1$.
From there, it will be convenient to define  $\varphi_\mu$ as the function whose graph is obtained by
rotating the graph of $f_\mu$ by $-\pi/4,$ flipping it and scaling it down by a factor $\sqrt 2$ and call it the RLM-excursion of the permuton $\mu$ 
(see Figure~\ref{fig:parametrization}, left).
The function $\varphi_\mu$ belongs to the space of {\em excursions}, defined as
\begin{equation}\label{eqn:defn_F}
	\mathcal{E} := \left\{ \varphi : [0,1] \to \R^+ \;:\; \mbox{$\varphi$ is $1$-Lipschitz and } \varphi(0)=\varphi(1)=0 \right\}.
\end{equation}


Our first main result for $231$ is the following:
\begin{prop}
\label{prop:intro-param231}
The map sending a $231$-avoiding permuton to its RLM-excursion is one-to-one and bicontinuous.
\end{prop}
We prove this result in \Cref{sec:structure_231} (see Lemmas~\ref{lem:231_Uniqueness} and~\ref{lem:bicontinuity}) by constructing explicitly the inverse map,
i.e., by explaining how to reconstruct a $231$-avoiding permuton given its RLM-curve.
Interestingly, this reconstruction procedure does not mimic the standard procedure
to reconstruct $231$-avoiding permutations from their RLM.
Indeed, the latter is recursive and difficult to imitate in the continuous setting.
Instead we provide a direct formula for the cumulative distribution function of the permuton,
see Lemma~\ref{lem:construction_muf}.
\medskip

We now turn our attention to the pattern $321$.
Again, it is well-known that $321$-avoiding permutations are determined by the positions and values of their
RLM (see e.g. \cite[Chapter 4]{BonaLivre}).
However, it is not the case for permutons: 
it turns out that there are multiple $321$-avoiding permutons with the same RLM-curve, as illustrated by \Cref{fig: counter-example RLM mass}, and that the map sending a $321$-permutons to its RLM-curve is not continuous. 
This prevents from encoding the set of $321$-avoiding permutons by the set $\mathcal E$ of~\eqref{eqn:defn_F}.

Instead, we will show that $321$-avoiding permutons can be parametrized by pairs $(\pi_1,\pi_2)$ of measures
living on the space
\begin{equation}
\label{def:D}
\mathcal D := \left\{ (\pi_1, \pi_2) \in \mathcal M([0,1])^2 \;:\;
\begin{array}{ll} \pi_1 \leq \Leb, \, \pi_2 \leq \Leb,\,
\pi_1([0,1]) = \pi_2([0,1])\\ \text{and }\forall x\in [0,1], \pi_1([0,x]) \le \pi_2([0,x])
\end{array}\right\},
\end{equation}
where $\mathcal M([0,1])$ is the space of finite measures on $[0,1]$, equipped with the topology of weak convergence. In other words, $\mathcal D$ is the space of pairs of measures on $[0,1]$ with the same total mass, both subuniform\footnote{
    A measure $\pi$ on $[0,1]$ is said to be {\em subuniform} if, for any  measurable $A \subseteq [0,1]$, one has $\pi(A)  \le \Leb(A)$.
},
and such that $\pi_1([0,x]) \le \pi_2([0,x])$ for all $x$ in $[0,1]$.
Informally, $\pi_1$ and $\pi_2$ are the projections on the horizontal and vertical axes
of the subdiagonal part of the 321-avoiding permuton $\mu$;
this is a continuous analogue of the counting measure of the positions and values of the right-to-left minima of a 321-avoiding permutation $\sigma$ (see Figure \ref{fig:parametrization}, right).

We shall prove the following statement, which is a simplified version of Proposition~\ref{prop:321-avoiding-permutons}.
\begin{prop}
\label{prop:intro-param321}
There exists an explicit continuous surjection $\Psi$ from $\mathcal D$ to the set of $321$-avoiding permutons.
\end{prop}
The description of $\Psi$ can be found in \Cref{sec:structure_321}.
We note that, unlike in the $231$-avoiding case,
the parametrization by $\mathcal D$ is not one-to-one.
Nevertheless, we will provide an explicit description of the fibers $\Psi^{-1}(\pi_1, \pi_2)$ (Proposition~\ref{prop:321-avoiding-permutons}),
and show that they are reduced to singletons whenever $\mu$ does not put any mass on the main diagonal of the unit square.

The above parametrizations of 231- and 321-avoiding permutons are illustrated on Figure~\ref{fig:parametrization}.
For further discussion on the differences between the two parametrizations, 
we refer the reader to Section \ref{sec:counterexamples}.
\begin{figure}[t]
\[\includegraphics[scale=.8]{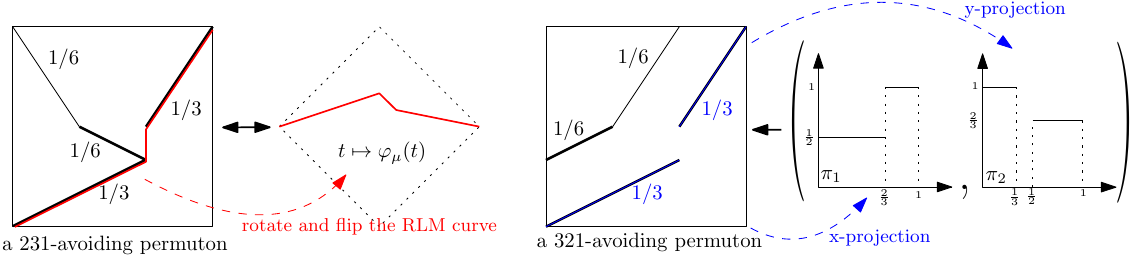}\]
\caption{Illustration of the parametrizations of 231- and 321-avoiding permutons.
321-avoiding permutons are in bijection with excursions, obtained by rotating and flipping their RLM curves (in red on the figure). On the other hand, there is a continuous surjection from pairs of measures to 231-avoiding permutons (the measures $\pi_1$ and $\pi_2$ are here represented by the graphs of their densities); informally, these measures are the $x$- and $y$-projections of the lower-diagonal part of the permuton (in blue on the figure).}
\label{fig:parametrization}
\end{figure}

\subsection{Second main results: large deviation principles for uniform 231 or 321-avoiding permutations}
\label{subseq:LDPintro}

As explained in Section~\ref{ssec:background}, an important step will be to establish LDPs for the random objects we are interested in, in the uniform case.

Roughly speaking, a sequence of random variables $(X_n)$ taking values in a Polish space $E$ satisfies an LDP at
speed $a_n$ with rate function $I$ if for any $x \in E,$ the probability $\mathbb P(X_n \simeq x)$ that $X_n$ is close to $x$ behaves as $\exp(-a_n I(x)),$ where $(a_n)$ is a sequence of real numbers going to infinity and $I$ is a lower semicontinuous nonnegative function (see Section \ref{ssec:LDP} for proper definitions).


An LDP for uniform random permutations in the space of permutons has been given and used in
\cite{trashorras2008large_deviations,mukherjee2016exponential_permutations,kenyon2020fixed_densities}.
An LDP for Mallows random permutations follows directly by a standard result
of large deviation theory, known as Varadhan's lemma; 
see, e.g., \cite[Theorem D.8]{anderson2010BookRandomMatrix}.
More recently, Borga, Das, Mukherjee and Winkler \cite{borga2024LDP_samples} have given an LDP for a large family 
of random permutations, obtained by sampling i.i.d.~random points from a probability measure in the plane.
In the present paper, we provide LDPs at speed $n$ for uniform $231$ (resp.~$321$)-avoiding permutations.
These LDPs take place respectively in the spaces $\mathcal E$ and $\mathcal D$ parametrizing $231$ and $321$-avoiding permutons introduced in the previous sections, and involve explicit  rate functions.
An LDP for random Mallows permutations conditioned to avoid $231$ (resp.~$321$) follows immediately (see Corollaries~\ref{cor:LDP_Mallows_321} and \ref{cor:LDP_Mallows_231}).
\medskip

Let us first describe the LDP that will be the cornerstone for the analysis of Mallows permutations avoiding $231.$ In view of Proposition \ref{prop:intro-param231}, it is natural to establish an LDP in the space of 1-Lipschitz functions. For a permutation $\sigma$ of size $n,$ let us denote by
$f_\sigma$ its normalized RLM-curve, defined as follows:

\[ f_\sigma(x) := \frac 1 n \left( \min\{\sigma(c) : c>nx\}-1\right)\]
and by $\varphi_\sigma$ the excursion obtained by rotating and flipping around the graph of $f_\sigma$ as on Figure~\ref{fig: example_231}.
By construction, $\varphi_\sigma$ is a continuous function on $[0,1]$, which is differentiable almost everywhere
with $\|\varphi_\sigma'\|\le 1.$ 
We endow the space $\mathcal E$ defined in~\eqref{eqn:defn_F} with the topology of the supremum norm, which makes it a compact space, and recall that all functions $\varphi\in\mathcal{E}$ are a.e.~differentiable on $[0,1]$ with derivative $|\varphi'| \le 1$.
Our LDP for uniform $231$-avoiding permutations reads as follows. 
For any $y \in [-1,1]$, we write
\begin{equation}\label{eqn:defn_J}
	J(y) := \frac12 (y+1)\log(y+1) + \frac12 (1-y)\log(1-y).
\end{equation}
\begin{prop}\label{prop:LDP_231_pi_intro}
  Let $H^{231}: \mathcal E \rightarrow [0,\infty]$ defined by
           \[H^{231}(\varphi) := 2 \int_{[0,1]} J(\varphi'(t)) \mathrm dt.\]
	For all $n \geq 1$, let $\sigma_{n,231}$ be a uniform random $231$-avoiding permutation of size $n$.
    Then the random variables $(\varphi_{\sigma_{n,231}})$ satisfy a large deviation principle at speed $n$ with rate function $H^{231}.$
\end{prop}

We now state our result for $321$.
As explained above and illustrated in more details in Section \ref{sec:counterexamples} below,
we cannot use the excursion $\varphi_\sigma$ to characterize $321$-avoiding permutons. 
At the level of LDPs, this will require establishing a different LDP in this case, although the underlying discrete objects are related to Dyck paths in both cases.
For a $321$-avoiding permutation $\sigma$, we let $(\pi_1^\sigma,\pi_2^\sigma)$ be the rescaled occupation measures of positions and values of its strict right-to-left minima, see \Cref{ssec:bij321} for a precise definition.
The pair $(\pi_1^\sigma,\pi_2^\sigma)$ is an element of the space $\mathcal D$ parametrizing $321$-avoiding permutons, which we defined in \eqref{def:D}.
Then, the following holds.
\begin{prop}\label{prop:LDP_321_pi_intro}
  Let $H^{321}: \mathcal D \rightarrow [0,\infty]$ be defined as follows:
\[
        H^{321}(\pi_1, \pi_2) := \int_0^1 \!\! \left( \rho_1 \log \rho_1 + (1\!-\!\rho_1) \log (1\!-\!\rho_1)
                + \rho_2 \log \rho_2 + (1\!-\!\rho_2) \log (1\!-\!\rho_2) + 2 \log 2 \right) \mathrm{dLeb} ,
    \]
     where $\rho_1$ and $\rho_2$ are the densities of $\pi_1$ and $\pi_2$ with respect to the Lebesgue measure\footnote{
     These densities do exist since elements of $\mathcal D$ are pairs of subuniform measures.
     }.

	For all $n \geq 1$, let $\sigma_{n, 321}$ be a random permutation chosen uniformly among $321$-avoiding permutations of size~$n$.
	Then, the associated pairs $\left( \pi_1^{\sigma_{n, 321}}, \pi_2^{\sigma_{n, 321}} \right)$ of measures
    satisfy a large deviation principle on $\mathcal D$ at speed $n$ with rate function $H^{321}.$
\end{prop}

\medskip

Note that in both cases, the rate functions $H^{231}$ and $H^{321}$ are explicit and simple, so that their analysis will be possible.
In particular, $H^{231}$ (resp. $H^{321}$) attains its unique minimum at $\varphi^* = 0$ (resp.~$\pi_1^* = \pi_2^* = \frac{1}{2} \textrm{Leb}_{|[0,1]}$), corresponding in both cases to the diagonal permuton. 
We retrieve the already known convergence of the uniform $231$ (resp. $321$)-avoiding permutations to this diagonal permuton. 
The fluctuations around this limit have been described in \cite{HoffmanBrownian1} and the  LDPs that we provide here complement these results naturally.

Also, the decay rate of single points probabilities, namely, $\mathbb P(\sigma_n(i)=j)$ when $i$ and $j$ are both of order $n$, have been studied in depth; see \cite{MinerPak,atapour2014largedev}.
Such results are neither weaker nor stronger than our LDPs. 
On the one hand,
we look at deviation probabilities of the whole object rather than the image of a single element, 
so it might seem that our results contain more information. 
However, on the other hand,
the topology on the space of permutons is too weak for the image of a specific element to be a continuous functional,
so one can not deduce deviation of single point probabilities from our LDPs.
\medskip

To conclude this section, let us say a few words of the proof. As pointed above, we do need to establish two LDPs respectively in the spaces $\mathcal E$ and $\mathcal D$. In both cases, the underlying random objects are Dyck paths, that is random excursions conditioned to stay nonnegative.
LDPs for random walks have been essentially known since the work of Mogulskii in the 70's \cite{mogulskij1976LDP_Random_Walks}. 
This involves a rate function which coincides with $H^{231}$ when it is finite and
we need to prove that the result can be transferred to the conditioned models.
Note that LDPs for continuous-time processes conditioned to return near zero have been obtained in the literature by \cite{mogulskii2014ldp,dort2024ldp}, but up to our knowledge, there are no known LDPs for Dyck paths.
The scheme of proof is similar for Propositions~\ref{prop:LDP_231_pi_intro} and~\ref{prop:LDP_321_pi_intro}, respectively.
The upper bound is easy as it follows immediately from the LDP for the unconditioned model,
and the fact that the conditioning event has a polynomially small probability (and not an exponentially small one).
More work is needed on the lower bound: we need to perform small transformations on our objects that force the conditioning event to hold, while not changing the probability too much.
This kind of transformation arguments is standard in the large deviation literature.

\subsection{Third main results: explicit limit shapes and asymptotic partition functions for Mallows random permutations conditioned to avoid 321 or 231}
\label{sec: permuton limits}

We now come back to our original motivation of studying constrained Mallows permutations.
Recall that for $\alpha\in\{321,231\}$ and $n\ge1$, we let $\tau_n^{\beta,\alpha}$ denote a Mallows permutation of size $n$ with bias parameter $q_n = \ee^{\beta/n}$ conditioned to avoid $\alpha$; see \eqref{eq:def_Constraint_Mallows}.

Our main theorems say that
the random permutations $\tau^{\beta,\alpha}_{n}$ converge to a deterministic permuton $\mu^\alpha_\beta$ which depends on $\beta$.
To describe the limit permutons, it is useful to use the language of push-forward measures. 
If $g:E \to F$ is a measurable function and 
$\mu$ a measure on $E$, then the push-forward measure $g_\# \mu$ 
is the measure defined by $g_\# \mu(A)=\mu(g^{-1}(A))$ for all measurable $A \subseteq F$. 
In the following, we implicitly use $X$ as the argument of $g$,
i.e.~we write $g(X)$ for the function $x \mapsto g(x)$. 
Lastly, $g_\#$ is seen as an operator on measures so that $(g_\#+h_\#)\mu$ is simply equal to $g_\# \mu+h_\# \mu$. In the results below, almost sure convergences hold for any coupling between the objects $\tau^{\beta,\alpha}_n$ for $n \geq 1$.

\begin{thm}
\label{thm:phase_transition_231}
For $\beta >0,$ we let
\[
    f^{231}_\beta(x) := 1 - \frac{1}{\beta} \log\left( 1 + \ee^\beta - \ee^{\beta x} \right) \,.
\]
Then, the random permutations $\tau^{\beta,231}_{n}$ converge almost surely in the sense of permutons to the deterministic permuton
\[
    \mu^{231}_\beta :=  \big(X,f^{231}_\beta(X)\big)_\# \min\left(1,(f^{231}_\beta)'(x) \right) \dd x
    + \big( X,1-X \big)_\# \max\left(0, 1- (f^{231}_\beta)'(x) \right) \dd x 
\]
if $\beta>0$,
and to the diagonal permuton if $\beta\le 0.$
\end{thm}

\begin{thm}
\label{thm:phase_transition_321}
For $\beta >0,$ we let
\[ f^{321}_\beta(x) := \frac12 + \frac1\beta \log\left( \frac{\ee^{\beta x} + \ee^{\beta/2}}{1 + \ee^{\beta/2} + \ee^{\beta} - \ee^{\beta x}} \right) \,.\]
Then, the random permutations $\tau^{\beta,321}_{n}$ converge almost surely in the sense of permutons to the deterministic permuton
\[ \mu^{321}_\beta := \big( (X,f^{321}_\beta(X))_\# + (f^{321}_\beta(X),X)_\#  \big) \frac{\dd x}{1+\exp\left(\beta(\frac12-x) \right)} 
\]
if $\beta>0$,
and to the diagonal permuton if $\beta\le 0.$
\end{thm}

\begin{figure}
\centering
\includegraphics[height=5cm]{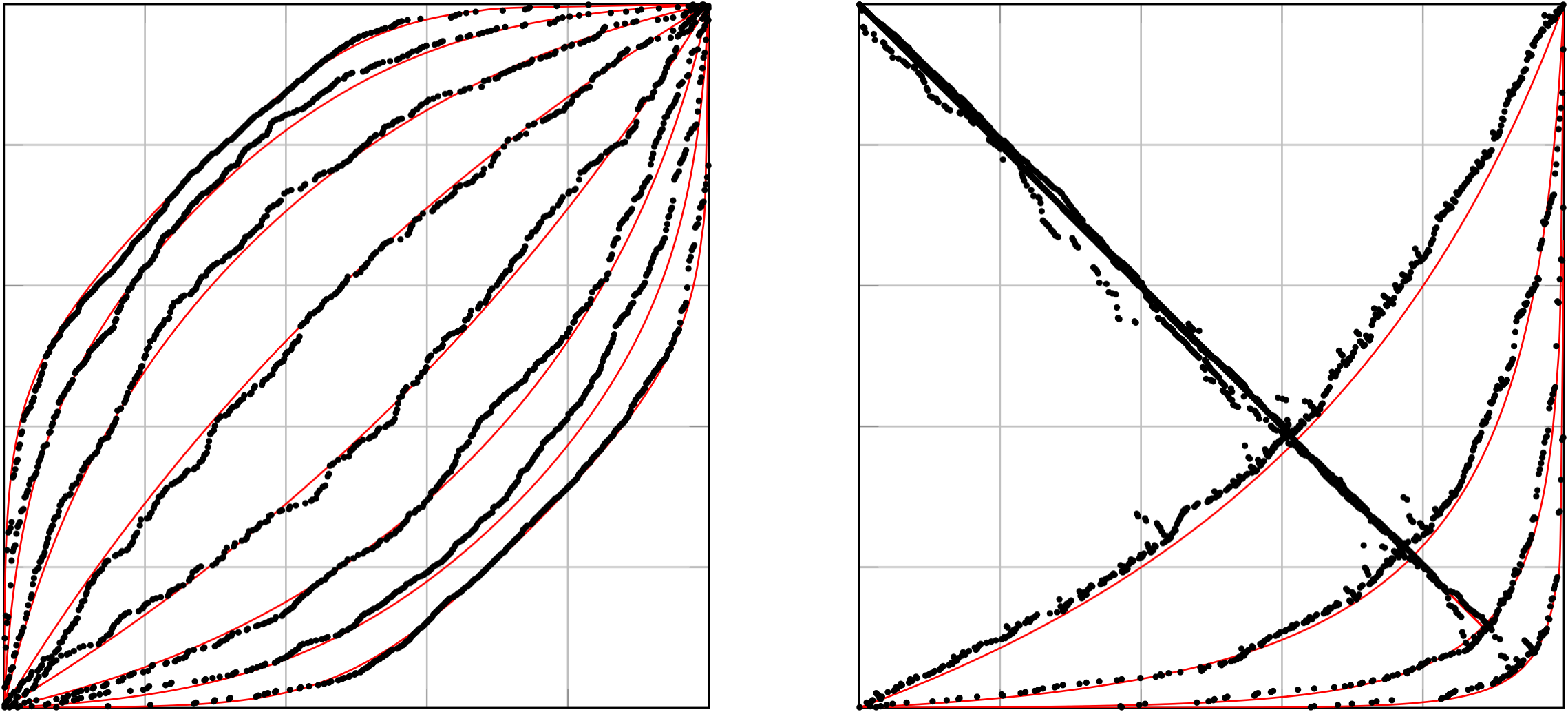}
\caption{Black dots: simulations of $\tau^{\beta,\alpha}_{n}$ with $n=600$ and $\beta\in\{1,3,6,12\}$.
Red curves: support of the limit permutons $\mu^{\alpha}_\beta$, again for $\beta\in\{1,3,6,12\}$.
Left: $\alpha=321$.
Right: $\alpha=231$.}
\label{fig:simulations_WithLimitCurves}
\end{figure}

In particular, we see that $\mu^{231}_\beta$ is supported on the union of an increasing curve
and a part of the decreasing diagonal, namely $\{(x,1-x): 0 \le x \le x_\beta^*\}$,
where $x_\beta^*$ fulfills $(f^{231}_\beta)'(x_\beta^*)=1$ 
(explicitly, $x_\beta^* = (\log(1+\ee^\beta)-\log(2) )/\beta$). On the other hand,   $\mu^{321}_\beta$ is supported on two increasing curves, symmetric of each other.
\Cref{fig:simulations_WithLimitCurves} shows realizations of $\tau^{\beta,321}_{n}$ and $\tau^{\beta,231}_{n}$ for various values of $\beta$, together with the support of their limit permutons.
\medskip
 
To finish this section, let us briefly discuss the proof strategy.
These results are obtained by using our LDPs for uniform $\alpha$-avoiding permutations, and then solving optimization problems. 
Informally speaking, our LDPs state that the probability that a uniform $\alpha$-avoiding permutation is close to a point $x$ in the permuton space is close to $\exp(-n I^\alpha(x))$.
Since all these permutations have roughly
$\frac12 n^2 \dens{21}{x}$ inversions 
(where $\dens{21}{x}$ is the density of pattern $21$ in $x$, see \Cref{ssec:permutons}), 
the probability that the Mallows random permutation $\tau^{\beta,\alpha}_{n}$ is close to $x$ is roughly proportional to
\[q_n^{n^2 \dens{21}{x}/2} \exp(-n I^\alpha(x)) = \ee^{n (\beta \dens{21}{x}/2 -I^\alpha(x))} \,.\]
The limit permuton will be the one that makes this quantity as large as possible,
i.e., that maximizes $\beta \dens{21}{x}/2 -I^\alpha(x)$. 
Note that these heuristics explain why $q_n=\ee^{\beta/n}$ was the right order of magnitude to have a non-trivial permuton limit.
Our proofs follow this heuristic, except that we formulate and solve the optimization
problems in the parameter spaces (namely,
 on excursions in the $231$-avoiding case and on pairs of measures in the $321$-avoiding case).
Remarkably, the density of inversions has a very simple expression in terms of these parameters, which eases the analysis.
In particular, the functions that we need to maximize are strictly convex,
and finding the optimum is a routine computation 
(even though the objects live in an infinite-dimensional space). 
This demonstrates the applicability of our LDPs.

As mentioned earlier, solving the above optimization problems also yield
asymptotic estimates for the underlying partition functions. 
We now state these results.
\begin{prop}\label{prop:estimateZ}
For any $\beta > 0,$ we have 
\begin{align}
	\lim_{n \rightarrow \infty} \frac{1}{n} \log (Z_n^{\beta,231}) &= 2 \log \left( 1+ \ee^\beta \right) -\beta  - \frac{1}{\beta} \int_{\ee^{-\beta}}^{\ee^\beta} \frac{\log y}{1+y} \dd y;\label{eq:partition231}\\
	\lim_{n \rightarrow \infty} \frac{1}{n} \log (Z_n^{\beta,321}) &= \frac{2}{\beta} \int_{\ee^{-\beta/2}}^{\ee^{\beta/2}} \frac{\log(1+y)}{y} \dd y.\label{eq:partition321}
\end{align}
For $\beta \le 0$, we have $\lim_{n \rightarrow \infty} \frac{1}{n} \log (Z_n^{\beta,231})=\lim_{n \rightarrow \infty} \frac{1}{n} \log (Z_n^{\beta,321})=\log 4$.
\end{prop}

We conclude this section with several comments.
\begin{rem}\label{rmk:negative_beta}
The last claim of Proposition~\ref{prop:estimateZ}, in the case $\beta \le 0$, is easy to prove. Indeed, let $C_n=\frac{1}{n+1} \binom{2n}n$ be the $n$-th Catalan number.
Since most of the $C_n$ many $\alpha$-avoiding permutations of size $n$
have $o(n^2)$ inversions, we have
	\[ C_n \, (1-o(1)) \, \ee^{- |\beta| o(n)} \le Z_n^{\beta,231}  \le C_n,\]
implying $\lim_{n \rightarrow \infty} \frac{1}{n} \log (Z_n^{\beta,321}) =\log 4$.
\end{rem}


\begin{rem}
When $\beta \to 0$, the permutons $\mu^{\alpha}_{\beta}$ converge to the diagonal permuton which is the limit of uniform $\alpha$-avoiding permutations as recalled above.
On the other hand, when $\beta \to \infty$, 
the permutons $\mu^{231}_{\beta}$ and $\mu^{321}_{\beta}$ converge respectively to the anti-diagonal permuton (in the $231$-avoiding case), 
and to the union of two increasing segments from $(0,1/2)$ to $(1/2,1)$ and from $(1/2,0)$ to $(1,1/2)$ (in the $321$-avoiding case). 
In both cases, it is easy to see that these are the $\alpha$-avoiding permutons maximizing the density of inversions, as expected.
\end{rem}

\begin{rem}
    Our results extend to all other patterns of size $3$ by symmetry.
    For instance, $123$-avoiding permutations are in one-to-one correspondance with $321$-avoiding permutations via $\sigma \mapsto \sigma^\mathrm{rev} := \sigma(n), \dots, \sigma(1)$.
    Through this map, right-to-left minima become left-to-right minima and inversions are swapped with non-inversions.
    Furthermore, the random permutations $\tau_n^{\beta,123}$ converge to the permuton $\mu_\beta^{123}$ obtained by flipping the permuton $\mu_\beta^{321}$ horizontally.
    In a similar way, $213$-avoiding permutations are in one-to-one correspondance with $231$-avoiding permutations via $\sigma \mapsto \sigma^\mathrm{comp} := n+1-\sigma(1), \dots, n+1-\sigma(n)$, and thus, the random permutations $\tau_n^{\beta,213}$ converge to the permuton $\mu_\beta^{213}$ obtained by flipping the permuton $\mu_\beta^{231}$ vertically.
    That way, we get limit shapes for $\tau_n^{\beta,\alpha}$ for all patterns $\alpha$ of size $3$.
\end{rem}

\subsection{Comparing 231 and 321-avoiding permutations with given RLM}
\label{sec:counterexamples}

In this section, we provide a few examples illustrating the differences between the $321$-avoiding and $231$-avoiding cases.

In our first family of examples, we consider two sets of right-to-left minima which differ by a single point. 
The associated pairs $\left( \pi_1, \pi_2 \right)$ of measures are then very close to each other, but the difference in the sets of right-to-left minima
have been chosen so that the associated RLM curves are far from each other.
In this case, the corresponding $321$-avoiding permutations converge to the same permuton, while the corresponding $231$-avoiding have different permuton limits 
--- see Figure~\ref{fig: counter-example RLM curve}.

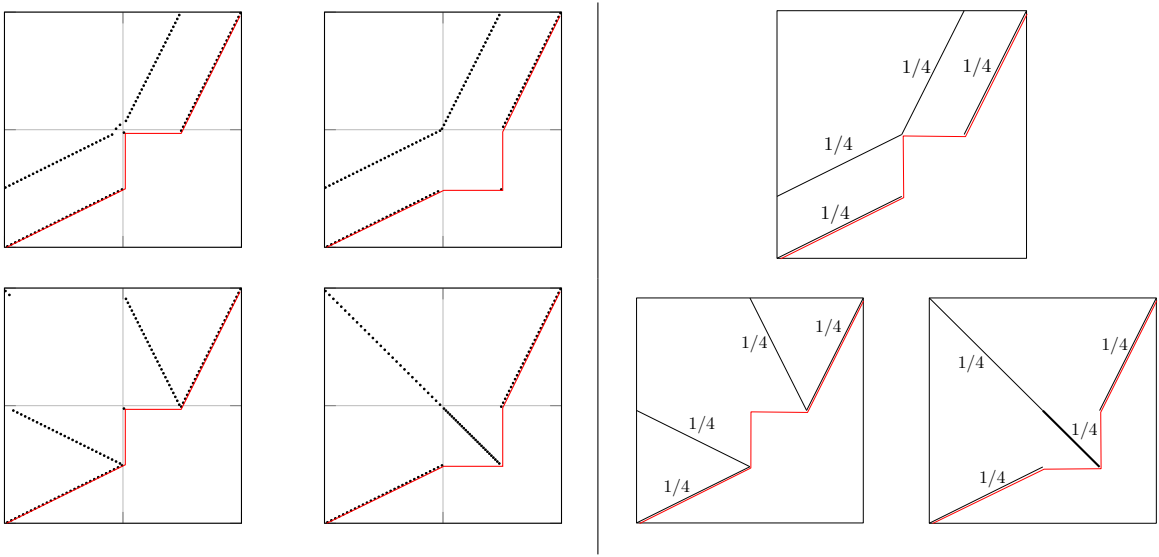
\begin{figure}
\[\hspace{-2mm}
\begin{array}{cc|cc}
\begin{array}{c}\begin{tikzpicture}
\begin{axis}[
      axis lines=box,
      xmin=0, xmax=1,
      ymin=0, ymax=1,
      xticklabels={}, 
      yticklabels={}, 
      xtick={0,0.5,1}, 
      ytick={0,0.5,1},
      grid=both,
      width=4.7cm,
      height=4.7cm
    ]
   \addplot[black, only marks, mark size = .3] table {321perm_curve1.dat};
   \end{axis}
          \draw[color=red] (0.03,0) -- (3.1*1/2+.04,3.1*1/4)--(3.1*1/2+.04,3.1*1/2-.04)--(3.1*3/4,3.1*1/2-.04)--(3.1,3.1-.03);
\end{tikzpicture}
\end{array}
&
\begin{array}{c}
\begin{tikzpicture}
\begin{axis}[
      axis lines=box,
      xmin=0, xmax=1,
      ymin=0, ymax=1,
      xticklabels={}, 
      yticklabels={}, 
      xtick={0,0.5,1}, 
      ytick={0,0.5,1},
      grid=both,
      width=4.7cm,
      height=4.7cm
    ]
   \addplot[black, only marks, mark size = .3] table {321perm_curve2.dat};
\end{axis}
          \draw[color=red] (0.03,0) -- (3.1*1/2+.02,3.1*1/4-.02)--(3.1*3/4+.02,3.1*1/4-.02)--(3.1*3/4+.02,3.1*1/2-.02)--(3.1,3.1-.03);
\end{tikzpicture}
\end{array}
&
 \hspace{2cm}\begin{array}{c}\includegraphics[height=3.3cm]{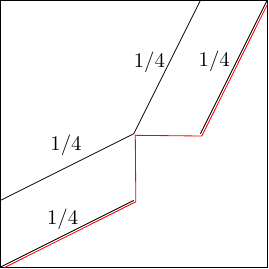}\end{array}\hspace{-2cm}
 &
\\
\begin{array}{c}
\begin{tikzpicture}
\begin{axis}[
      axis lines=box,
      xmin=0, xmax=1,
      ymin=0, ymax=1,
      xticklabels={}, 
      yticklabels={}, 
      xtick={0,0.5,1}, 
      ytick={0,0.5,1},
      grid=both,
      width=4.7cm,
      height=4.7cm
    ]   \addplot[black, only marks, mark size = .3] table {231perm_curve1.dat};
\end{axis}
          \draw[color=red] (0.03,0) -- (3.1*1/2+.04,3.1*1/4)--(3.1*1/2+.04,3.1*1/2-.04)--(3.1*3/4,3.1*1/2-.04)--(3.1,3.1-.03);
\end{tikzpicture}
\end{array}
&
\begin{array}{c}
\begin{tikzpicture}
\begin{axis}[
      axis lines=box,
      xmin=0, xmax=1,
      ymin=0, ymax=1,
      xticklabels={}, 
      yticklabels={}, 
      xtick={0,0.5,1}, 
      ytick={0,0.5,1},
      grid=both,
      width=4.7cm,
      height=4.7cm
    ]   \addplot[black, only marks, mark size = .3] table {231perm_curve2.dat};
\end{axis}
       \draw[color=red] (0.03,0) -- (3.1*1/2+.02,3.1*1/4-.02)--(3.1*3/4+.02,3.1*1/4-.02)--(3.1*3/4+.02,3.1*1/2-.02)--(3.1,3.1-.03);
\end{tikzpicture}
\end{array}
 &
 \begin{array}{c}\includegraphics[height=3cm]{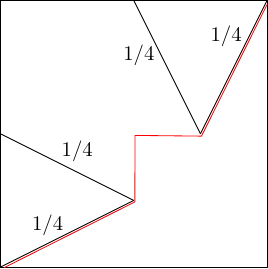}\end{array}
 &
   \begin{array}{c}  \includegraphics[height=3cm]{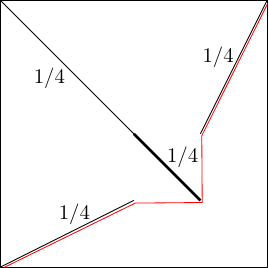}\end{array}
\end{array}\]
\caption{Examples of $\alpha$-avoiding permutations of size $120$ with given right-to-left minima on the left, and their limit permutons on the right.
In red, we plot the RLM curves of the permutations/permutons (slightly shifted so that they do not hide the points).
In the top line $\alpha=321$, while in the bottom line $\alpha=231$.
First column: RLMs are the points $(2i, i)$ for $i< \frac14n$, $(\frac12n, \frac12n)$, and $(n-i, n-2i)$ for $i<\frac14n$.
Second column: we have the same RLMs except that $(\frac12n, \frac12n)$
is replaced by $(\frac34n, \frac14n)$}
\label{fig: counter-example RLM curve}
\end{figure}

Our second family of examples illustrates the opposite situation.
Here, we choose two sets of right-to-left minima such that the associated RLM curves
are close to each other, but the associated pairs of measures are very different
(because one contains twice as many RLMs as the other).
In this case, the corresponding $231$-avoiding permutations converge to the same permuton,
but not the corresponding $321$-avoiding permutations ---
see Figure~\ref{fig: counter-example RLM mass}.

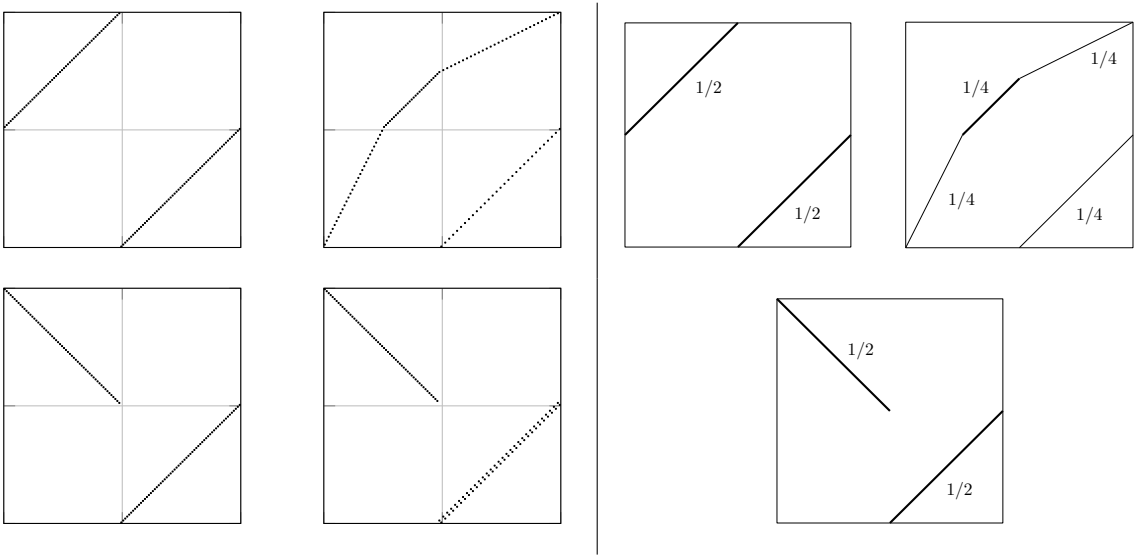
\begin{figure}
\[\hspace{-3mm} \begin{array}{cc|cc}
 \begin{array}{c}
\begin{tikzpicture}
\begin{axis}[
      axis lines=box,
      xmin=0, xmax=1,
      ymin=0, ymax=1,
      xticklabels={}, 
      yticklabels={}, 
      xtick={0,0.5,1}, 
      ytick={0,0.5,1},
      grid=both,
      width=4.7cm,
      height=4.7cm
    ]   \addplot[black, only marks, mark size = .2] table {321perm_mass1.dat};
\end{axis}
\end{tikzpicture}\end{array}
&
 \begin{array}{c}
\begin{tikzpicture}
\begin{axis}[
      axis lines=box,
      xmin=0, xmax=1,
      ymin=0, ymax=1,
      xticklabels={}, 
      yticklabels={}, 
      xtick={0,0.5,1}, 
      ytick={0,0.5,1},
      grid=both,
      width=4.7cm,
      height=4.7cm
    ]   \addplot[black, only marks, mark size = .2] table {321perm_mass2.dat};
\end{axis}
\end{tikzpicture}\end{array}
&
 \begin{array}{c}\includegraphics[height=3cm]{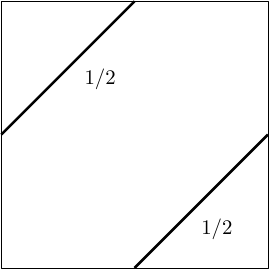}\end{array}&
     \begin{array}{c}\includegraphics[height=3cm]{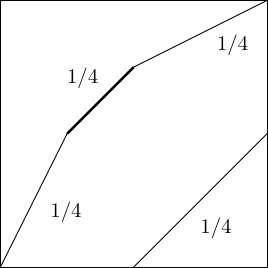}\end{array}\\
 \begin{array}{c}
\begin{tikzpicture}
\begin{axis}[
      axis lines=box,
      xmin=0, xmax=1,
      ymin=0, ymax=1,
      xticklabels={}, 
      yticklabels={}, 
      xtick={0,0.5,1}, 
      ytick={0,0.5,1},
      grid=both,
      width=4.7cm,
      height=4.7cm
    ]   \addplot[black, only marks, mark size = .2] table {231perm_mass1.dat};
\end{axis}
\end{tikzpicture}\end{array}
&
 \begin{array}{c}
\begin{tikzpicture}
\begin{axis}[
      axis lines=box,
      xmin=0, xmax=1,
      ymin=0, ymax=1,
      xticklabels={}, 
      yticklabels={}, 
      xtick={0,0.5,1}, 
      ytick={0,0.5,1},
      grid=both,
      width=4.7cm,
      height=4.7cm
    ]   \addplot[black, only marks, mark size = .2] table {231perm_mass2.dat};
\end{axis}
\end{tikzpicture}\end{array}&
 \hspace{20mm}\begin{array}{c}
\includegraphics[height=3cm]{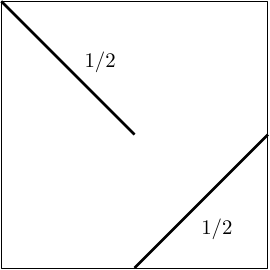}
\end{array} \hspace{-20mm}\end{array}\]
\caption{Further examples of $\alpha$-avoiding permutations of size $120$ with given right-to-left minima
on the left, and their limit permutons on the right.
In the top line $\alpha=321$, while in the bottom line $\alpha=231$.
First column: RLMs are the points $(\frac12n+i, i)$ for $i\le \frac12n$.
Second column: we keep only half of the RLMs, namely the RLMs are the points $(\frac12n+2i, 2i)$ for $i\le \frac14n$.}
\label{fig: counter-example RLM mass}
\end{figure}

These two families of examples illustrate that the RLM curve is a continuous parametrization for 231-avoiding permutons, while the description via pair of measures is more appropriate for 321-avoiding objects.

\subsection{Outline of the paper}
We end this introduction with a brief description of the organization of the paper. In the next section, we gather some notations and preliminaries on the various objects (LDPs, permutons, etc) that we use. 
The rest of the paper, from Section \ref{sec:structure_321} on, will be split in two parts, respectively devoted to the study of $321$ and of $231$-avoiding permutations. 
Note that, as the description of the results was easier for the $231$ case, we chose to start with the description of the corresponding results in the introduction. 
However, as the proofs are more involved in this case, we chose to devote the first part of the core of the paper to the $321$ case. As both parts can be read more or less independently, we hope it will be convenient for the reader.

\section{Notation and preliminaries}

For a compact Polish space $X$ equipped with its Borel $\sigma$-algebra, we let $\mathcal M(X)$ be the space of measures with finite mass on $X,$ endowed with the topology of weak convergence.

For two measures $\mu_1$ and $\mu_2$, we write $\mu_1 \le \mu_2$ to mean that $\mu_1(A) \le \mu_2(A)$ for any Borel set $A.$
Also, we write $f_{\#}\mu$ for the push-forward measure of $\mu$ by $f$, defined by $f_{\#}\mu (A)=\mu(f^{-1}(A))$
for any Borel set $A.$

\subsection{Large deviation principles}
\label{ssec:LDP}

In this subsection, we give a short introduction to large deviation principles (LDP).
This is standard material.
We refer the reader to Appendix D in~\cite{anderson2010BookRandomMatrix} for a concise introduction of useful tools on LDPs, and to~\cite{dembo1998large-deviations} for a thorough reference.

Let $E$ be a Polish space, $(X_n)$ a sequence of $E$-valued random variables, and $(a_n)$ a sequence of positive real numbers tending to $\infty$.
Let $I$ be a nonnegative lower semicontinuous function on $E$, possibly taking value $\infty$.
We say that the sequence $(X_n)$ satisfies an LDP\footnote{In the standard references \cite{dembo1998large-deviations} and \cite{anderson2010BookRandomMatrix}, LDPs are defined for sequences of probability measures and not of random variables. 
However, as often in probability theory, we shall identify random variables with their distribution, and say that a sequence of random variables satisfies an LDP when their distributions do.} 
at speed $a_n$ with rate function $I$ if, for any closed set $F \subseteq E$, one has
\begin{equation}\label{eq:def_LDP_ub}
  \limsup_{n \to \infty} \tfrac1{a_n} \log\mathbb{P}(X_n \in F) \le -\inf_{x \in F} I(x),
\end{equation}
and for any open set $O \subseteq E$, one has
\begin{equation}\label{eq:def_LDP_lb}
   \liminf_{n \to \infty} \tfrac1{a_n} \log\mathbb{P}(X_n \in O) \ge -\inf_{x \in O} I(x).
\end{equation}

Moreover, we say that the sequence $(X_n)$ satisfies a weak LDP if~\eqref{eq:def_LDP_lb} holds for open sets and if the upper bound~\eqref{eq:def_LDP_ub} holds for compact sets \cite[Definition D.2]{anderson2010BookRandomMatrix}.
Note that if $E$ is a compact set, weak and full LDP are equivalent. If moreover $d$ is a metric on $E$ (giving the required topology), and $B(x,\delta) := \{ y\in E : d(x,y) \le \delta\},$ then weak LDPs follow from the two following bounds~(see~\cite[Corollary D.6]{anderson2010BookRandomMatrix}):
\[ \lim_{\delta \rightarrow 0} \limsup_{n \rightarrow \infty} \tfrac1{a_n} \log\mathbb{P}(X_n \in B(x,\delta)) \le - I(x),  \]
and
\[ \lim_{\delta \rightarrow 0} \liminf_{n \rightarrow \infty} \tfrac1{a_n} \log\mathbb{P}(X_n \in B(x,\delta)) \ge - I(x).\medskip  \]

The following property, known as the contraction principle, is a useful tool to transfer LDPs, see e.g.~\cite[Theorem 4.2.1]{dembo1998large-deviations}.
\begin{prop}
 \label{prop:contractionprinciple}
 Assume that $(X_n)$ satisfies an LDP at speed $a_n$ with good\footnote{A rate function is said to be \textit{good} if its level sets $\{x \in E: \, I(x) \le M\}$ are compact.} rate function $I$ on some space $E$.   
 Let $F:E \rightarrow G$ be a bounded continuous function. 
 Then, the sequence of random variables $Y_n=F(X_n)$ satisfies an LDP at speed $a_n$ with good rate function $J$ given, for any $y\in G,$ by
 \[ J(y) =  \inf\{ I(x) : F(x)=y\}.\]
\end{prop}

We will also use the following lemma, that can be seen as a corollary of Varadhan's lemma
(see \cite[Theorem 4.3.1]{dembo1998large-deviations} and~\cite[Theorem D.8]{anderson2010BookRandomMatrix}).
\begin{lem}
 \label{lem:Varadhan}
 For any $n\ge 1,$ let $\mu_n$ denote the distribution of $X_n$ and assume that $(X_n)$ satisfies an LDP at speed $a_n$ with good rate function $I.$ Let $F:E \rightarrow \mathbb R$ be a bounded continuous function. Let the distribution $\nu_n$  be defined as follows:
 \[ \mathrm{d}\nu_n(x) := \frac{1}{Z_n} \ee^{a_nF(x)} \mathrm{d}\mu_n(x),\]
 where $Z_n$ is a normalizing constant such that $\nu_n$ is a probability measure.
 Then we have
	\[ \lim_{n \rightarrow \infty} \frac{1}{a_n} \log Z_n = -  \inf_{y\in E} (I(y)- F(y)). \]
	Moreover, if we let $Y_n$ have distribution $\nu_n$,
the sequence $(Y_n)$ satisfies an LDP at speed $a_n$ with good rate function $J$
 given by
 \[ \forall x\in E, \quad J(x) = I(x) - F(x) - \inf_{y\in E} (I(y)- F(y)).\]
\end{lem}

The last point explains how to deduce almost sure convergence from an LDP. 
It is an easy consequence of the Borel--Cantelli lemma.
\begin{lem}
 \label{lem:asconvergence}
 Assume that $(X_n)$ satisfies an LDP at speed $n$ with good rate function $I$ and that $I$ has a unique minimizer $x^*.$
 Then $(X_n)$ converges almost surely to $x^*$.
\end{lem}

\subsection{Permutons}
\label{ssec:permutons}
Let $\mathcal{P}$ be the set of permutons, that is, of probability measures $\mu$ on $[0,1]^2$ satisfying
\[ \mu([0,x] \times [0,1]) =\mu([0,1] \times [0,x])=x\]
for every $x$ in $[0,1]$ (i.e., $\mu$ has uniform marginals).
We endow the set of permutons with the topology of weak convergence.
If $\sigma$ is a permutation of size $n$, we associate with it a permuton $\mu_\sigma$ which attributes mass $1/n$ uniformly to each cell 
\[
    C_{i}^\sigma := [(i-1)/n, i/n]  \times \left[ (\sigma(i)-1) / n, \sigma(i)/n \right]  ,
    \quad 1\le i\le n \,.
\]
We say that a sequence $(\sigma_n)$ of permutations, converges to a permuton $\mu$ if $\mu_{\sigma_n}$ converges to $\mu$ weakly.
It is sometimes useful to define slight variations of $\mu_\sigma$:
we let $\mu_\sigma^\diagup$ (resp.~$\mu_\sigma^\diagdown$) be the permuton which attributes mass $1/n$ uniformly on the \emph{diagonal} (resp.~\emph{antidiagonal}) of each cell $C_{i}^\sigma$.
Since the Kolmogorov--Smirnov distances between these three permutons are bounded by $\cO(1/n)$, convergence results and large deviation principles hold for all three simultaneously.

If $\sigma$ is a permutation of size $n$ and $\alpha$ a permutation of size $k \le n$, the number of occurrences of $\alpha$ in $\sigma$ is
\[
    \occ(\alpha, \sigma) := \card{ (i_1, \dots, i_k) : 1\le i_1< \dots< i_k\le n \mbox{ and } \sigma(i_{\alpha^{-1}(1)})< \dots< \sigma(i_{\alpha^{-1}(k)}) }
\]
and the density of the pattern $\alpha$ in $\sigma$ is
\[
    \dens{\alpha}{\sigma} := \binom{n}{k}^{-1} \occ(\alpha, \sigma) \,.
\]
We say that the permutation $\sigma$ is $\alpha$-avoiding if $\dens{\alpha}{\sigma}=0$.

Similarly, if $\alpha$ is again a permutation of given size $k,$ we define the density of the pattern $\alpha$ in a permuton $\mu$ as
\[
    \dens{\alpha}{\mu} := \int k! \One_{x_1< \dots< x_k} \One_{y_{\alpha^{-1}(1)}< \dots< y_{\alpha^{-1}(k)}} \dd\mu(x_1, y_1) \dots \dd\mu(x_k, y_k) \,,
\]
and we say that the permuton $\mu$ is $\alpha$-avoiding if $\dens{\alpha}{\mu}=0$.
Note that if $\mu_\sigma$ is the permuton associated to the permutation $\sigma,$ we have in general $\dens{\alpha}{\sigma} \neq \dens{\alpha}{ \mu_\sigma}$. Nonetheless, it is easy to see that $\dens{321}{\sigma} = \dens{321}{\mu^\diagup_\sigma}$ and that $\dens{231}{\sigma} = \dens{231}{\mu^\diagdown_\sigma}$.
In particular, a permutation $\sigma$ is $321$-avoiding if and only if $\mu_\sigma^\diagup$ is $321$-avoiding, and $\sigma$ is $231$-avoiding if and only if $\mu_\sigma^\diagdown$ is $231$-avoiding.
Moreover, we know from \cite{Permutons} that the weak convergence of permutons is equivalent to the convergence of all pattern densities. 
Standard arguments show that a permuton $\mu$ is $\alpha$-avoiding if and only if there exists a sequence $(\sigma_n)$ of $\alpha$-avoiding permutations which converges to $\mu$.


\subsection{Nondecreasing subsets and coupling}
\label{ssec: nondecreasing subsets and coupling}

A subset $S$ of $\mathbb R^2$ will be called nondecreasing if one has $(x_1-x_2)(y_1-y_2) \ge 0$
for any $(x_1,y_1)$ and $(x_2,y_2)$ in $S$.
A coupling of two real-valued random variables $X$ and $Y$ is called nondecreasing
if the support of $(X,Y)$ is nondecreasing.
Given distributions $\mu$ and $\nu$ on $\mathbb R$, 
there is a unique nondecreasing coupling $(X,Y)$ of $X$ with law $\mu$ and $Y$ with law $\nu$.
Let us denote by $\mu \nearrow \nu$ the corresponding distribution of $(X,Y)$.

More precisely, if we define
\begin{equation}\label{eqn:defn_quantile_function}
    G_\mu : u\in(0,1) \mapsto \sup\{ x\in\R : \mu([0,x]) \le u \}
\end{equation}
the quantile function of $\mu$ and similarly $G_\nu$ the quantile function of $\nu$, 
we can write
\[ \mu \nearrow \nu = (G_\mu, G_\nu)_\# \textrm{Leb}_{[0,1]}. \]

We extend the notation to finite measures with the same total mass $M$
by setting
\[ \mu \nearrow \nu := M \cdot \left( \left( \frac{\mu}{M} \right) \nearrow \left( \frac{\nu}{M} \right) \right) .\]

Finally if $\mu$ is a measure on $[0,1]^2$, we denote by $\proj(\mu)$ the pair $(\mu_1,\mu_2)$,
where $\mu_1$ and $\mu_2$ are the projections of $\mu$ on the horizontal and vertical axes, respectively.
By construction it holds that $\proj(\mu_1 \nearrow \mu_2)=(\mu_1,\mu_2)$ 
and conversely if $\proj(\pi)=(\mu_1,\mu_2)$ and if $\pi$ has a nondecreasing support, 
then necessarily $\pi=\mu_1 \nearrow \mu_2$.

\part{Avoiding 321}
\label{part:321}

The goal of this part is to study the limit shape of $321$-avoiding random Mallows permutations with parameter $q_n = \ee^{\frac{\beta}{n}}.$ As explained in the introduction, we will successively establish Proposition~\ref{prop:intro-param321}, Proposition~\ref{prop:LDP_321_pi_intro}, Equation~\ref{eq:partition321} and Theorem~\ref{thm:phase_transition_321}. 
To ease the notations, we will drop the superscript $321$ in all the notation in this part.

\section{Permutations and permutons avoiding 321}
\label{sec:structure_321}

\subsection{A classical bijection for finite objects}
\label{ssec:bij321}
We say that a pair $(a,b)$ is a strict right-to-left (RL) minimum of a permutation $\sigma$ if $b=\sigma(a)<a$ and for any $c>a$, we have $\sigma(c)>b$ .
In the literature, RL minimum sometimes refers to the position $a$, or to the value $b$;
here we chose to call RL minimum the pair $(a,b)$.
In $321$-avoiding permutations, a pair $(a,b)$ is a strict RL minimum if and only if $b=\sigma(a)<a$.

For any permutation $\sigma$ of size $n,$ we denote by  $((a_1,b_1), \dots, (a_k,b_k))$  the set of its RL minima in increasing order,
that is
\begin{equation}\label{eq:condition_AB}
	1 <  a_1 < \dots <a_k \le n,
	\quad
	1 \le  b_1 < \dots <b_k < n,
	\quad\text{ and } a_i >b_i\text{ for all }1\le i \le k.
	\end{equation}
We also define
\[
\textrm{Dyck}_{n}:= \left\{(A,B)\in \mathcal P_n^2 : |A|=|B| \textrm{ and } \forall i \le |A|, \, a_i >b_i \right\},
\]
where $\mathcal P_n$ is the set of all subsets of $[n]= \{1, \ldots , n\}$ and $a_i$ (resp. $b_i$) is the $i$th smallest element of $A$ (resp. $B$).
Note that the choice of the notation $\textrm{Dyck}_{n}$ is related to the fact that this set is also in bijection with Dyck paths with $2n$ steps.

The following lemma holds, see e.g.~\cite[Section 4]{callan2007bijectionsdyckpaths321avoiding}.
\begin{lem}\label{lem:bij321}
	The application $\sigma \mapsto ((a_1,b_1), \dots, (a_k,b_k))$ which maps a permutation $\sigma$ to its list of strict RL minima in increasing order is a one-to-one correspondence between the set of  $321$-avoiding permutations of size $n$ and the set $\textrm{Dyck}_{n}.$
\end{lem}
Indeed, given the positions and values of the strict RL minima, which need to form an increasing
subsequence below the diagonal, one can complete the permutation
in a unique way by a second increasing subsequence weakly above the diagonal.
See the left part of Figure~\ref{fig:321-avoiding}.

For our purposes, it will be useful to encode these two lists as two measures, namely, if $((a_1,b_1), \dots, (a_k,b_k))$ is the list of  RL minima of the permutation $\sigma$ of size $n,$ we define
\[\pi_1^\sigma :=  \sum_{j=1}^k \Leb_{\left[\frac{a_j-1}n, \frac{a_j}n\right]} ,\qquad
\pi_2^\sigma :=  \sum_{j=1}^k \Leb_{\left[\frac{b_j-1}n, \frac{b_j}n\right]} .\]
This encoding turns out to be handy to compute the number of inversions of the permutation.
In the following, we write $\inv(\sigma)$ for the number of inversions of a permutation $\sigma$.

\begin{figure}
    \centering
    \includegraphics{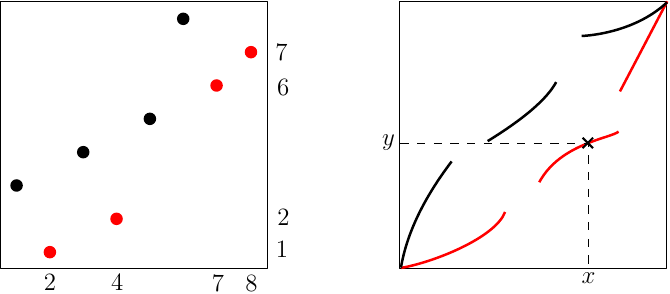}
    \caption{Left: a $321$-avoiding permutation $\sigma$, with its strict RL minima in red.
    The corresponding list encoding is $((2,1),(4,2),(7,6),(8,7))$. 
    Note that the element $(5,5)$ is a weak RL minimum 
    (in the sense that it is a fixed point of the permutation) 
    and does not appear in the list.
    Right: the support of a $321$-avoiding permuton $\mu=\Psi(\pi_1,\pi_2)$. 
    The support of $\pi_1 \nearrow \pi_2$ is in red.
    A point $(x,y)$ on the red curve should satisfy $\pi_1([0,x]) = \pi_2([0,y])$.
    For  $(\pi_1, \pi_2)$ in $\mathcal D$, we assume $\pi_1([0,x]) \le \pi_2([0,x])$, which forces $x \ge y$, i.e.~the support of $\pi_1 \nearrow \pi_2$ lies in the closed lower-right triangle $\{(x,y):x \ge y\}$.
    \label{fig:321-avoiding}}
\end{figure}

\begin{lem}\label{lem:inv_sigma}
    Let $\sigma$ be a $321$-avoiding permutation, and $((a_1,b_1),\dots,(a_k,b_k))$ be its list of strict RL minima. 
    Let also $\pi_1^\sigma$ and $\pi_2^\sigma$ be the associated measures on $[0,1]$,
    and $\rho_1^\sigma$ and $\rho_2^\sigma$ be their densities.
    We have:
    \[\inv(\sigma)=\sum_{i=1}^k (a_i-b_i) = n^2 \int_0^1 x \left( \rho^\sigma_1(x)-\rho^\sigma_2(x) \right) \mathrm{d}x \,.\]
\end{lem}
\begin{proof}
    Since $\sigma$ is $321$-avoiding, the elements in $\sigma$ that are not RL minima necessarily form an increasing subsequence. 
    It follows that, in each inversion of $\sigma$, the lower right element is an RL minimum. We claim that the number of inversions
    whose lower right element is a given RL minimum $(a_i,b_i)$ is $(a_i-b_i)$. Indeed, since $(a_i,b_i)$ is an RL minimum, the number of elements above and on the left of it
    is equal to the number $a_i-1$ of elements on its left, minus the number $b_i-1$ of elements below it.
    Hence the total number of inversions in $\sigma$ is $\sum (a_i-b_i)$, where the sum is taken over RL minima of $\sigma$, 
    or equivalently over strict RL minima (since fixed points contribute $0$ to the sum).
    This justifies the first equality.
    
    For the second equality, we simply observe that
    \[\int_0^1 x \rho^\sigma_1(x) \mathrm{d}x = \int_0^1 x \, \pi^\sigma_1(\mathrm{d}x) = \frac{1}{n}  \sum_{i=1}^k \frac{a_i-1/2}{n},\]
    and a similar equality holds replacing $\rho^\sigma_1$ and $\pi^\sigma_1$ by $\rho^\sigma_2$ and $\pi^\sigma_2$, and $a_i$ by $b_i$.
\end{proof}

\subsection{Parametrization of 321-avoiding permutons}

The goal of this section is to give a concrete parametrization of $321$-avoiding permutons, extending in some sense the bijection of Lemma~\ref{lem:bij321}.
For this it is useful to decompose permutons as follows:
given a permuton $\mu$ in $\mathcal P$, we denote by $\mu_{\lrtriangle}$, $\mu_{\ultriangle}$ and $\mu_{\diagup}$
the restrictions of $\mu$ respectively to the lower-right triangle $\{(x,y): x>y\}$, to the upper-left triangle $\{(x,y): x<y\}$,
and to the diagonal $\{(x,y): x=y\}$.
We start with a lemma.

\begin{lem}\label{lem:incr_supports}
Let $\mu$ be a $321$-avoiding permuton.
The supports of $\mu_{\lrtriangle} + \mu_{\diagup}$ and  $\mu_{\ultriangle} + \mu_{\diagup}$ are nondecreasing subsets of $[0,1]^2$.
\end{lem}

\begin{proof}
    Assume for the sake of contradiction that there exist $(x_1,y_1)$ and $(x_2,y_2)$ in the support of $\mu_{\lrtriangle} + \mu_{\diagup}$ that satisfy $x_1 <x_2$, but $y_1>y_2$. 
    By construction, this support is a subset of the closed triangle $\{(x,y): x\ge y\}$, so we also have $x_1 \ge y_1$ (and $x_2 \ge y_2$).
    Since $\mu$ is a permuton, it holds
    that
    \[ 
        \mu([0,x_1] \times [0,1]) =x_1 \ge y_1 =\mu([0,1] \times [0,y_1]) \,.
    \]
    Subtracting $\mu([0,x_1] \times [0,y_1])$ to both sides, we have
    \[ 
        \mu([0,x_1] \times [y_1,1]) \ge \mu([x_1,1] \times [0,y_1]) \,.
    \]
    But the right-hand side is nonzero, since $(x_2,y_2)$ is in the support of $\mu$.
    Therefore there exists a point $(x_0,y_0)$ satisfying $x_0 <x_1$ and $y_0>y_1$ in the support of $\mu$.
    The existence of three points $(x_0,y_0)$, $(x_1,y_1)$ and $(x_2,y_2)$ in decreasing position in the support of $\mu$ contradicts the fact that $\mu$ is $321$-avoiding.
    This proves that the support of $\mu_{\lrtriangle} + \mu_{\diagup}$ is nondecreasing.
    The proof that the support of $\mu_{\ultriangle} + \mu_{\diagup}$ is nondecreasing is similar.
\end{proof}


We recall that the space $\mathcal D$ has been defined in \eqref{def:D}.
Note that $\mathcal D$ is a compact and convex subset of $\mathcal M([0,1])^2$.
It is straightforward to check that, for any $321$-avoiding permutation $\sigma$, the pair $\left( \pi_1^\sigma,\pi_2^\sigma \right)$ defined above is in $\mathcal{D}$.
For any $(\pi_1,\pi_2)$ in $\mathcal D$, the measure
\[
    \Psi(\pi_1,\pi_2) := \pi_1 \nearrow \pi_2 +(\Leb_{[0,1]}-\pi_1) \nearrow (\Leb_{[0,1]}-\pi_2) \,,
\]
where the operation $\nearrow$ has been defined in \Cref{ssec: nondecreasing subsets and coupling}, is a permuton supported on a union of two nondecreasing sets and is thus $321$-avoiding.
See the right part of Figure~\ref{fig:321-avoiding}.

\begin{prop}\label{prop:321-avoiding-permutons}
    The map $\Psi$ is a continuous surjective map from $\mathcal D$ to the set $\PAvTDU$ of $321$-avoiding permutons.
    Moreover, considering a $321$-avoiding permuton $\mu$ and its decomposition $\mu=\mu_{\lrtriangle}+\mu_{\ultriangle}+\mu_{\diagup}$ as above, we have
    \begin{equation}\label{eq:preimages}
        \Psi^{-1}(\mu) = \big\{\proj(\mu_{\lrtriangle} + \nu) : 0 \le \nu \le \mu_{\diagup} \big\} \,,
    \end{equation}
    where we recall that $\proj(\pi)$ is the pair of measures given by the projections of $\pi$ on the horizontal and vertical axes respectively.
\end{prop}

\begin{proof}
    We first prove that $\Psi$ is a continuous mapping. 
    Consider a sequence $(\pi^n_1,\pi^n_2)$ converging weakly to $(\pi_1,\pi_2)$. 
    Formally, we can write 
    \[\pi^n_1 \nearrow \pi^n_2= (G_1^n,G_2^n)_{\#}\Leb_{[0,M_n]},\]
    where $G_1^n$ and $G_2^n$ are the quantile functions of $\pi^n_1$ and $\pi^n_2$,
    and $\Leb_{[0,M_n]}$ is the Lebesgue measure on the interval $[0,M_n]$ (of total mass $M_n$). If $M=\pi_1([0,1])= \pi_2([0,1]),$
    the weak convergence of $\pi_1^n$ to $\pi_1$ implies that $M_n=\pi_1^n([0,1])$ tends to $M$, and consequently $\Leb_{[0,M_n]}$ converges in total variation distance to $\Leb_{[0,M]}$.
    Moreover, we have the convergence of the quantile functions $G^n_1$, $G^n_2$ to the quantile functions $G_1$ and $G_2$ of $\pi_1$ and $\pi_2$, at least at continuity points of $G_1$, $G_2$.
    Since $G_1$, $G_2$ have at most countably many discontinuity points, Lemma~\ref{lem:cv_push_forward} applies, and $\pi^n_1 \nearrow \pi^n_2$ converges weakly to $\pi_1 \nearrow \pi_2$.
    Similarly, $(\Leb_{[0,1]}-\pi^n_1) \nearrow (\Leb_{[0,1]}-\pi^n_2)$ converges weakly to $(\Leb_{[0,1]}-\pi_1) \nearrow (\Leb_{[0,1]}-\pi_2)$, which proves continuity.
    
    We now prove \eqref{eq:preimages}, which will imply the surjectivity of $\Psi$, since the right-hand side is always nonempty.
    Let $\mu$ be a $321$-avoiding permuton, and $(\pi_1,\pi_2)$ a pair in $\mathcal D$ such that $\Psi(\pi_1,\pi_2)=\mu$.
    The condition \enquote{$\pi_1([0,x]) \le \pi_2([0,x])$ for all $x$} implies that
    $\pi_1 \nearrow \pi_2$ is supported on the {\em closed} lower-right triangle $\{(x,y): x \ge y\}$ (see Figure~\ref{fig:321-avoiding}),
    while $(\Leb_{[0,1]}-\pi_1) \nearrow (\Leb_{[0,1]}-\pi_2)$ is supported on the closed upper-left triangle $\{(x,y): x \le y\}$.
    Hence we have
    \[  \mu_{\lrtriangle} \le \pi_1 \nearrow \pi_2 \le  \mu_{\lrtriangle} + \mu_{\diagup}. \]
    Hence, there is a measure $\nu \leq \mu$ such that $\pi_1 \nearrow \pi_2=\mu_{\lrtriangle}+\nu$, so $(\pi_1,\pi_2)=\proj(\mu_{\lrtriangle} + \nu)$.
    
    It remains to prove the reverse inclusion. 
    Let $\nu \le \mu_{\diagup}$, and set $(\pi_1,\pi_2)=\proj(\mu_{\lrtriangle} + \nu)$. 
    By Lemma~\ref{lem:incr_supports}, the support of $\mu_{\lrtriangle} + \mu_{\diagup}$ is nondecreasing, and this also holds for $\mu_{\lrtriangle} + \nu$. 
    By uniqueness of the nondecreasing coupling, we have
    \begin{equation}\label{eq:lower_part}
        \mu_{\lrtriangle} + \nu=\pi_1 \nearrow \pi_2.
   \end{equation}
    Similarly, we have $(\Leb_{[0,1]}-\pi_1,\Leb_{[0,1]}-\pi_2)=\proj(\mu_{\ultriangle} + (\mu_{\diagup}-\nu))$, and $\mu_{\ultriangle} + (\mu_{\diagup}-\nu)$ has a nondecreasing support, so
    \begin{equation}\label{eq:upper_part}
        \mu_{\ultriangle} + (\mu_{\diagup}-\nu)=(\Leb_{[0,1]}-\pi_1) \nearrow (\Leb_{[0,1]}-\pi_2).
   \end{equation}
   Combining equations \eqref{eq:lower_part} and \eqref{eq:upper_part} gives $\mu=\Psi(\pi_1,\pi_2)$, as wanted.
\end{proof}

\section{Large deviation principle for 321-avoiding permutations}

We recall that
\[
\mathcal D := \left\{ (\pi_1, \pi_2) \in \mathcal M([0,1])^2 :
\begin{array}{ll} \pi_1 \leq \Leb, \pi_2 \leq \Leb,
\pi_1([0,1]) = \pi_2([0,1])\\ \text{and }\forall x\in [0,1], \pi_1([0,x]) \le \pi_2([0,x])
\end{array}\right\}
\]
and define
\[
\mathcal S := \{ \pi \in \mathcal M([0,1]) : \pi \leq \Leb\} 
\]
the space of subuniform measures on $[0,1]$. 
The space $\mathcal S$ can be endowed with the Kolmogorov distance
\[ d(\pi_1, \nu_1) := \sup_{x\in [0,1]} |\pi_1([0,x]) - \nu_1([0,x])|,\]
which induces the topology of weak convergence (since measures in $\mathcal S$ have no atoms).
We extend this to $\mathcal S^2$ (and thus $\mathcal D$) by taking the maximum of the distances between coordinates.
For $\pi \in \mathcal S^2$ and $\varepsilon >0,$ we denote by
\[ B(\pi, \varepsilon) := \{ \nu \in \mathcal S^2 : d(\pi, \nu) \le \varepsilon\}\]
the ball of center $\pi$ and radius $\varepsilon.$

Our goal is to prove the LDP for $321$-avoiding permutations, that is, Proposition~\ref{prop:LDP_321_pi_intro}.
We start by introducing some notation. In the following, if $(\sigma_n)$ is a sequence of permutations, we write $\left( \pi_1^n, \pi_2^n \right)$ instead of $\left( \pi_1^{\sigma_n}, \pi_2^{\sigma_n} \right).$
Let $\mathcal P_n$ be the set of subsets of $\{1,\dots,n\}$. For $A \in \mathcal{P}_n$, we let $\nu_n^A$ be the measure $\sum_{a \in A}  \Leb_{[\frac{a-1}{n}, \frac a n )}$, which lives in $\mathcal S$.

By Lemma~\ref{lem:bij321}, the set $\mathrm{Dyck}_{n}$ is in bijection with $321$-avoiding permutations of size $n.$
Therefore, if $\sigma_n$ is a uniform random $321$-avoiding permutation of size $n$, then $\left( \pi_1^n, \pi_2^n \right)$ has the same distribution as $(\nu_n^A,\nu_n^B)$,
where $(A,B)$ is taken uniformly at random in $\mathrm{Dyck}_{n}$.
We will compare this model with a related but simpler model:
let $A'$ and $B'$ be two uniform random subsets of $\{1,\dots,n\}$,
independent from each other. 
Note that $(\nu_n^{A'},\nu_n^{B'})$ belongs to $\mathcal S^2$ but not necessarily to $\mathcal D.$

We define the continuous function $F_n^{A'}$ which is affine on each interval $[\frac{a-1}{n}, \frac a n ]$ and such that
\[F_n^{A'}(x)=\frac{1}{n} \sum_{i=1}^{nx} \One[i \in A']\]
when $xn$ is an integer. 
Then $F_n^{A'}$ is a normalized nondecreasing random walk with $\mathrm{Bernoulli}(1/2)$ steps.
From \cite[Theorem 5.1.2]{dembo1998large-deviations}, it follows
that $F_n^{A'}$ satisfies a large deviation principle\footnote{
    \cite[Theorem 5.1.2]{dembo1998large-deviations} is in fact stated for a càdlàg version of $F_n$,
    but both versions are exponentially equivalent \cite[Lemma 5.1.4]{dembo1998large-deviations}.
    Hence, large deviation principles for both versions are equivalent~\cite[Theorem 4.2.13]{dembo1998large-deviations}.
} at speed $n$ with rate function
$H_1(f)= \int_{0}^1 I(f'(x)) \mathrm{d}x$,
where 
\[I(y) := \begin{cases}  y \log(y)+(1-y)\log(1-y) +\log(2) & \text{if $0 \le y \le 1$};\\
\infty &\text{otherwise}\end{cases}\] is the Fenchel--Legendre transform of the moment generating function of a 
$\mathrm{Bernoulli}(1/2)$ random variable.
Note that $H_1(f)=\infty$ unless $f'$ takes its values in $[0,1]$ almost everywhere. But a function $f$ with $0 \le f' \le 1$ a.e.~can be identified with the measure $f' \mathrm{d}\Leb$, and the $L^\infty$ topology on such functions coincides with the Kolmogorov distance
on the set $\mathcal S$ of subuniform measures.
Hence the above large deviation principle can be rephrased as a large deviation principle on measures: if $A'$ is chosen uniformly at random in $\mathcal{P}_n$, then the sequence $(\nu_n^{A'})$
 satisfies a large deviation principle  on the set $\mathcal S$ at speed $n$ with rate function $H_1$ where
\[ H_1(\pi)= \int_{0}^1 \rho \log(\rho )+(1-\rho )\log(1-\rho )+\log(2) \]
if $\pi$ has density $\rho.$

Since $\nu_n^{B'}$ is an independent copy of $\nu_n^{A'}$,
it follows that the pair $\left( \nu_n^{A'},\nu_n^{B'}\right)_{n \ge 1}$
satisfies a large deviation principle on $\mathcal S^2$ at speed $n$ with rate function $H(\pi'_1,\pi'_2)=H_1(\pi'_1)+H_1(\pi'_2)$,
which coincides with the definition given in Proposition \ref{prop:LDP_321_pi_intro}.
We note for later use that the same large deviation principle holds if we take $A'$ and $B'$ to be uniform random subsets of $\{1,\dots,n-1\}$
(indeed, the two models are exponentially equivalent in the sense of \cite[Definition 4.2.12]{dembo1998large-deviations}).

\begin{proof}[Proof of Proposition~\ref{prop:LDP_321_pi_intro}]
As explained in Section \ref{ssec:LDP}, as $\mathcal D$ is compact, it is enough to show that, for any $\pi \in \mathcal D,$
\begin{equation}\label{LDP:ub}
 \lim_{\varepsilon \to 0} \limsup_{n \to \infty} \frac{1}{n} \log \mathbb P\big( \left( \pi_1^n, \pi_2^n \right) \in B(\pi, \varepsilon) \big) \le - H(\pi),
\end{equation}
and
\begin{equation}\label{LDP:lb}
 \lim_{\varepsilon \to 0} \liminf_{n \to \infty} \frac{1}{n} \log\mathbb P\big( \left( \pi_1^n, \pi_2^n \right) \in B(\pi, \varepsilon)  \big) \ge - H(\pi).
\end{equation}
We start with the upper bound \eqref{LDP:ub}, which is easy to establish.
The number of Dyck paths of length $2n$ is the Catalan number $C_n = \frac{1}{n+1}\binom{2n}{n}.$ From the Stirling formula, it is easy to show that $C_n=\frac{1}{\sqrt{\pi}}n^{-3/2}4^n(1+o(1)).$ Consequently, for any  $\pi \in \mathcal D,$ we have
\begin{eqnarray*}
 P\big( \left( \pi_1^n, \pi_2^n \right) \in B(\pi, \varepsilon)  \big)& = &
 \frac{1}{C_n}  \#  \left\{ (A,B) \in  \mathrm{Dyck}_n : (\nu_n^A, \nu_n^B) \in B(\pi, \varepsilon)\right\}\\
& \le & \frac{\#  \left\{ (A',B') \in \mathcal P_{n}^2 : (\nu_n^{A'}, \nu_n^{B'}) \in B(\pi, \varepsilon)\right\}}{4^{n}} \frac{4^{n} }{C_n} \\
& \le  &   \mathbb P \left( \left(\nu_n^{A'}, \nu_n^{B'}\right) \in B(\pi, \varepsilon) \right) \, n^{3/2} \, (\sqrt{\pi}+o(1)).
 \end{eqnarray*}
Therefore,
\[\! \lim_{\varepsilon \to 0} \limsup_{n \to \infty} \frac{1}{n} \log\mathbb P\left( \left( \pi_1^n, \pi_2^n \right) \in B(\pi, \varepsilon\right)) \le \lim_{\varepsilon \to 0} \limsup_{n \to \infty} \frac{1}{n} \log \mathbb P\left( \left( \nu_n^{A'}, \nu_n^{B'}\right) \in B(\pi, \varepsilon)\right) \le - H(\pi),\]
where the second inequality is given by the large deviation principle for $(\nu_n^{A'}, \nu_n^{B'})$.
If $\pi \in \mathcal S^2 \setminus \mathcal D,$ for $\varepsilon$ small enough $B(\pi, \varepsilon) \cap \mathcal{D} = \emptyset$ so that
$ \mathbb P\left( \left( \pi_1^n, \pi_2^n \right) \in B(\pi, \varepsilon\right)) =0.$

We now prove the lower bound. Let $\pi = (\pi_1, \pi_2) \in \mathcal D$. 
We denote by $F_1(x) = \pi_1([0,x])$ and   $F_2(x) = \pi_2([0,x])$ the respective cumulative distribution functions of $\pi_1$ and $\pi_2$.
By definition of $\mathcal D$, for all $ x \in [0,1]$, we have $0 \le F_1(x) \le F_2(x) \le x$ and $F_1(1)=F_2(1)$.
Consider a pair $(A',B')$ in $\mathcal P_{n-1}^2$ (note the shift of index, in particular $n \notin B'$) such that $(\nu_n^{A'},\nu_n^{B'})$ is in $B(\pi ,\varepsilon)$.
The goal is to associate with it a pair $(A,B)$ in $\Dyck_n$ such that $(\nu_n^A,\nu_n^B)$ is also close to $\pi$.

First observe that $(\nu_n^{A'},\nu_n^{B'}) \in B(\pi ,\varepsilon)$ implies that, for all $1\le k \le n,$
\[ \left| \frac{1}{n} \sum_{a \in A'} \One[a \le k] - F_1\left(\frac{k}{n}\right)\right| \le \varepsilon,\]
and
\[ \left| \frac{1}{n} \sum_{b \in B'} \One[b \le k] - F_2\left(\frac{k}{n}\right)\right| \le \varepsilon.\]
Hence, using $ F_1\left(\frac{k}{n}\right)\le  F_2\left(\frac{k}{n}\right),$ we get that
\begin{equation}\label{eq:discrepancy}
 \min_{1 \le k \le n} \left( \sum_{b \in B'} \One[b < k] - \sum_{a \in A'} \One[a < k] \right) \ge - 2\varepsilon n.
\end{equation}
Also, since $F_1(1)=F_2(1)$, setting $\delta=|B'|-|A'|$, we have $|\delta| \le 2 \varepsilon n$.
We now define $A$ as follows:
\begin{itemize}
    \item If $|A'| > 2 \varepsilon n$, we first let $A''$ be obtained from $A'$ by removing its
    $\lfloor 2\varepsilon n \rfloor +1$ smallest elements. Then we let $A$ be obtained from $A''$ by adding
    the $\lfloor 2\varepsilon n \rfloor+\delta+1$ largest elements of $\{1,\dots,n\}$ which are not in $A''$.
    \item If $|A'| \le 2 \varepsilon n$, we let $A=\{n-|B'|+1,\dots,n\}$.
\end{itemize}
Finally simply set $B=B'$. 
We claim that $(A,B)$ is in $\Dyck_n$. Let us prove it in the first case above, the second case being easier. 
That $|A|=|B|$ holds by construction.
To check that $a_i > b_i$ for all $i \le |A|$, we verify the equivalent assertion: for all $k \le n$
\begin{equation}\label{eq:ToCheck_A_Dominates_B}
    \sum_{b \in B} \One[b < k] \ge \sum_{a \in A} \One[a < k],
\end{equation}
with a strict inequality whenever $k \in A$.
We set $\bar{a}=\max(A' \setminus A'')$ and $\tilde{a}=\min(A \setminus A'')$, i.e.~these are respectively the largest element that has been removed and the smallest element that has been added to go from $A'$ to $A$.
If $\bar{a} \ge \tilde{a}$, then $A$ consists of the $|B|$ largest elements in $\{1,\dots,n\}$, and $a_i>b_i$ holds trivially for all $i \le |A|$ (recall that $n \notin B$ by construction).
So assume $\bar{a} < \tilde{a}$, and let us prove \eqref{eq:ToCheck_A_Dominates_B}.
\begin{itemize}
    \item If $k \le \bar{a}$, then $\sum_{a \in A} \One[a \le k]=0$
    and \eqref{eq:ToCheck_A_Dominates_B} holds trivially.
    \item If $k \ge \tilde{a}$, then $A$ contains every element larger than $k$. Consequently, we have
    \[\sum_{a \in A} \One[a \ge k]=n-k+1 > \sum_{b \in B} \One[b \ge k],\]
    since $n \notin B$. Since $|A|=|B|$, this implies \eqref{eq:ToCheck_A_Dominates_B} with a strict inequality.
    \item Finally, if $\bar{a} < k < \tilde{a}$, then we have
    \[\sum_{a \in A} \One[a \le k]= \sum_{a \in A'} \One[a \le k] -  \lfloor 2\varepsilon n \rfloor-1.\]
    Indeed, the $\lfloor 2\varepsilon n \rfloor +1$ elements of the set difference $A' \setminus A$ are all smaller than $k$,
    while the elements of the  set difference $A \setminus A'$ are all larger than $k$.
    The strict version of inequality~\eqref{eq:ToCheck_A_Dominates_B} then follows from \eqref{eq:discrepancy}.
\end{itemize}
Thus \eqref{eq:ToCheck_A_Dominates_B} holds for all $k$ (with a strict inequality when $k \in A$), and $(A,B)$ lies in $\Dyck_n$ as claimed.

Moreover, since we have transformed $A'$ into $A$ by removing at most $ 2\varepsilon n +1$ and adding at most $2\varepsilon n+\delta+1$ elements, we have
\[d(\nu_n^A,\nu_n^{A'}) \le 4\varepsilon+\tfrac{\delta+2}n \le 7 \eps\]
for $n$ sufficiently large since $\delta \le 2 \eps n$. 
Let us call $T$ the map associating $(A,B)$ with $(A',B')$. 
A preimage $(A',B')$ of $(A,B)$ is determined by the set of $k \le 7\eps n$ elements that have been added to or removed from $A'$ to obtain $A$. 
Choose $\varepsilon$ small enough so that $7\varepsilon <1/2$. 
Then, for a given $(A,B) \in \textrm{Dyck}_n$, the number of its preimages by $T$ is at most
\[\sum_{k \le 7 \eps n} \binom{n}{k} \le 7 \eps n \binom{n}{7\eps n} \le \frac{ e \sqrt{7  \eps n}}{2\pi \sqrt{1-7\eps}} ( (7\eps)^{7\eps} (1-7\eps)^{1-7\eps})^{-n},\]
where we used that $k \mapsto \binom{n}{k}$ is increasing if $k<n/2$ and $k!\ee^k k^{-k-1/2} \in [\sqrt{2\pi},e]$ for all $k \ge 1$.

Since $T$ maps any $(A',B')$ in $\mathcal P_{n-1}^2$ with $(\nu_n^{A'},\nu_n^{B'}) \in B(\pi ,\varepsilon)$ to a pair $(A,B)$ in $\textrm{Dyck}_n$ with $(\nu_n^A,\nu_n^B) \in B(\pi , 8 \varepsilon)$, we have
\begin{multline}
    \# \{(A,B) \in \textrm{Dyck}_n : (\nu_n^A,\nu_n^B) \in B(\pi, 8\varepsilon)\} \ge
\frac{2\pi \sqrt{1-7\eps}}{ e \sqrt{7 \eps n}} ( (7\eps)^{7\eps} (1-7\eps)^{1-7\eps})^n \\
\cdot \# \{(A',B') \in \mathcal P_{n-1}^2: (\nu_n^{A'},\nu_n^{B'}) \in B(\pi ,\varepsilon)\}.
\end{multline}
Therefore, if $(A,B)$ and $(A',B')$ are uniformly distributed  in $\Dyck_n$ and $\mathcal P_{n-1}^2$ respectively,
\begin{multline*}
    \liminf_{n \to \infty} \frac 1 n \log \mathbb P\left( (\nu_n^A,\nu_n^B) \in B(\pi, 8\varepsilon) \right) \ge 7\eps \log(7\eps) +(1-7\eps)\log(1-7\eps) \\ + \liminf_{n \to \infty} \frac 1 n \log\mathbb P\left( (\nu_n^{A'},\nu_n^{B'}) \in B(\pi ,\varepsilon) \right) .
\end{multline*}
Letting $\eps$ go to $0$ and using the large deviation principle for $(\nu_n^{A'},\nu_n^{B'})$ give \eqref{LDP:lb}. 
This concludes the proof of Proposition \ref{prop:LDP_321_pi_intro}.
\end{proof}

From there, it is straightforward to deduce the corresponding LDP for the pair of measures associated
with a random Mallows permutation, conditioned to avoid 321. 
Omitting the superscript $321,$ let us denote by $\tau_n^{\beta}$
a $321$-avoiding Mallows random permutation with parameters $n \geq 1$ and $q_n = \ee^{\frac{\beta}{n}}.$
To lighten notation, we write $(\pi_1^{[n,\beta]}, \pi_2^{[n,\beta]}) := \left( \pi_1^{\tau_n^{\beta}}, \pi_2^{\tau_n^{\beta}} \right)$.
Finally, for $\beta \in \R$, we define the action $A_{\beta}(\pi_1, \pi_2)$ on $\mathcal{D}$ by
\begin{equation}\label{eqn:defn_action_321}
A_{\beta}(\pi_1, \pi_2) := H(\pi_1, \pi_2) - \beta \int_0^1 x \left( \rho_1(x)-\rho_2(x) \right) \mathrm{d}x,
\end{equation}
where $\rho_1$ and $\rho_2$ are the respective densities of $\pi_1$ and $\pi_2$. Then, we have the following LDP.
\begin{cor}\label{cor:LDP_Mallows_321}
    Let $\beta \in \R$.
    The pair $(\pi_1^{[n,\beta]}, \pi_2^{[n,\beta]})$ satisfies a large deviation principle on $\mathcal D$ at speed $n$ with rate function
    \[ A_{\beta}(\pi_1,\pi_2)-\inf(A_{\beta}).\]
\end{cor}

\begin{proof} By Lemma \ref{lem:inv_sigma},  the law of $(\pi_1^{[n,\beta]}, \pi_2^{[n,\beta]})$ is absolutely continuous with respect to the law of $(\pi_1^{[n,0]}, \pi_2^{[n,0]})$, with Radon--Nikodym derivative given by
    \[ \frac{Z_n^0}{Z_n^{\beta}} \exp \left( \frac{\beta}{n} \inv(\sigma) \right) = \frac{Z_n^0}{Z_n^{\beta}} \exp \left( \beta n \int_0^1 x(\rho_1(x)-\rho_2(x)) \right) \mathrm{d}x, \]
    where $\rho_1$ and $\rho_2$ are the respective densities of $\pi_1$ and $\pi_2$ with respect to the Lebesgue measure.
    The corollary now follows from
    Proposition~\ref{prop:LDP_321_pi_intro} and Varadhan's Lemma \ref{lem:Varadhan}.
\end{proof}

The limit shape will be determined by minimizing the function $A_\beta$; this is detailed in the next section. Before getting there, we make a remark on these LDPs.
For this, we observe that, by construction, for any 321-avoiding permutation $\sigma$, we have $\mu_{\sigma}^\diagup=\Psi(\pi_1^\sigma,\pi_2^\sigma)$.
Since $\Psi$ is a continuous mapping,
using  the contraction principle for large deviations (Proposition~\ref{prop:contractionprinciple}), it is possible to deduce from Proposition \ref{prop:LDP_321_pi_intro} an LDP in the space of permutons. Although we won't use it in the sequel, we think it may be interesting \emph{per se}.

\begin{cor}\label{cor:LDP_uniform_321}
    For all $n \geq 1$, let $\sigma_n$ be a uniform random $321$-avoiding permutation of size $n$.
    Then, the permutons $\mu_{\sigma_n}^\diagup$ satisfy a large deviation principle at speed $n$ with rate function
    \[\overline{H}(\mu) := \left\{
    \begin{array}{ll}
      \inf_{0 \le \nu \le \mu_{\diagup}} H(\proj(\mu_{\lrtriangle}+\nu)) & \textrm{if } \mu \in  \PAvTDU,\\
      \infty & \textrm{otherwise.}
    \end{array}
\right.\]
The same LDP holds for the permutons $\mu_{\sigma_n}.$
\end{cor}
The last statement holds because the laws of $\mu_{\sigma_n}$ and $\mu_{\sigma_n}^\diagup$
are exponentially equivalent to each other in the sense of~\cite[Definition 4.2.10]{dembo1998large-deviations}.

\section{An optimization problem for pairs of measures}

Our goal is now to minimize the action $A_{\beta}$ appearing in Corollary~\ref{cor:LDP_Mallows_321}.
\begin{prop}\label{prop:max_action}
	For any $\beta\in \R$, the action $A_{\beta}$ has a unique minimizer $(\pi_1^{(\beta)}, \pi_2^{(\beta)})$ in $\mathcal{D}$. For $\beta \leq 0$, we have $\pi_1^{(\beta)}=\pi_2^{(\beta)}=\frac{1}{2}\Leb_{[0,1]}$. 
    For $\beta>0$, the densities of $\pi_1^{(\beta)}$ and $\pi_2^{(\beta)}$ on $[0,1]$ are given by
	\begin{equation}\label{eqn:formula_maximum_action}
	    \rho_1^{(\beta)}(x) := \frac{1}{1+\exp\left(\beta(\frac12-x) \right)}, 
        \quad \rho_2^{(\beta)}(x) := \frac{1}{1+\exp\left(\beta(x-\frac12) \right)}.
	\end{equation}
	In particular, both $\pi_1^{(\beta)}$ and $\pi_2^{(\beta)}$ have total mass $\frac{1}{2}$.
\end{prop}

\begin{proof}[Proof of Proposition~\ref{prop:max_action}]
	We will rely on differential calculus on the space $\mathcal{D}$ (see Proposition~\ref{prop:max_concave} in Appendix~\ref{sec:appendix} for a precise statement).
	We first note that the domain $\mathcal{D}$ is convex and compact for the topology of weak convergence of measures. 
	Moreover, the action $A_{\beta}$ is a lower semicontinuous function of $(\pi_1, \pi_2)$. 
	Finally, the entropy $H$ is a strictly convex function of $(\pi_1, \pi_2)$ and the additional term in the action is affine, so $A_{\beta}$ is strictly convex, which shows the existence and uniqueness of the minimizer.
	For $\beta \leq 0$, the term $H$ has a unique minimum at $\pi_1=\pi_2=\frac{1}{2}\Leb_{[0,1]}$, and the second term of~\eqref{eqn:defn_action_321} is always nonnegative and is $0$ for $\pi_1=\pi_2$, which identifies the minimum.
	
	We now assume $\beta>0$. By Proposition~\ref{prop:max_concave}, to ensure that the minimizer of $A_{\beta}$ is given by the $(\pi_1^{(\beta)}, \pi_2^{(\beta)})$ of~\eqref{eqn:formula_maximum_action}, it is now sufficient to show that the differential of $A_{\beta}$ vanishes at $(\pi_1^{(\beta)}, \pi_2^{(\beta)})$.
	
	To this end, we define the set of directions
	\[V_{\left(\pi_1^{(\beta)}, \pi_2^{(\beta)}\right)}(\mathcal{D}):=\big\{(\mu_1, \mu_2):(\pi_1^{(\beta)}+t\mu_1,\pi_2^{(\beta)}+t\mu_2) \in \mathcal D \text{ for $t>0$ small enough}\big\}\]
	in which the derivative could make sense.
	We first check that any $(\mu_1, \mu_2) \in V_{\left(\pi_1^{(\beta)}, \pi_2^{(\beta)}\right)}(\mathcal{D})$ is a pair of signed measures with
    $\mu_1([0,1])=\mu_2([0,1])$ and such that there is a constant $C$ such that $|\mu_1|,|\mu_2| \leq C \Leb$.	
	Indeed, one must have $\mu_1([0,1])=\mu_2([0,1])$ to ensure that $\pi_1^{(\beta)}([0,1])+t\mu_1([0,1])=\pi_2^{(\beta)}([0,1])+t\mu_2([0,1])$ for $t$ small enough. 
	Moreover, for $i \in \{1,2\}$, the condition $|\mu_i| \leq C \Leb$ is necessary to ensure that $0\le \pi_i^{(\beta)}+t\mu_i \le \Leb$ for $t$ small enough.

    We now differentiate the action $A_{\beta}$. For any direction $(\mu_1, \mu_2) \in V_{(\pi_1^{(\beta)}, \pi_2^{(\beta)})}(\mathcal{D})$, writing $\alpha_1, \alpha_2$ for the densities of $\mu_1, \mu_2$, we have
    \[ \frac{\mathrm{d}}{\mathrm{d}t}\Big|_{t=0^+} H \left( \pi_1^{(\beta)}+t\mu_1, \pi_2^{(\beta)}+t\mu_2 \right)= \int_0^1 \alpha_1 \log \frac{\rho_1^{(\beta)}}{1-\rho_1^{(\beta)}} + \alpha_2 \log \frac{\rho_2^{(\beta)}}{1-\rho_2^{(\beta)}}. \]
    Note that differentiating inside the integral sign is indeed legit since the densities $\rho_1^{(\beta)}$ and $\rho^{(\beta)}_2$ are bounded away from $0$ and $1$, and hence everything is uniformly bounded for $t$ small enough.
    On the other hand, the second term of the action $A_{\beta}$ is a linear function so we can easily differentiate it. 
    We obtain:
    \begin{multline}\label{eqn:derivative_action}
        \frac{\mathrm{d}}{\mathrm{d}t}\Big|_{t=0^+} A_{\beta} \left( \pi_1^{(\beta)}+t\mu_1, \pi_2^{(\beta)}+t\mu_2 \right) \\= \int_0^1 \alpha_1 \log \frac{\rho_1^{(\beta)}}{1-\rho_1^{(\beta)}} + \alpha_2 \log \frac{\rho_2^{(\beta)}}{1-\rho_2^{(\beta)}} -\beta \int_0^1 x \left( \alpha_1(x)-\alpha_2(x) \right) \mathrm{d}x.
    \end{multline}
    Replacing $\rho_1^{(\beta)}$ and $\rho_2^{(\beta)}$ by their explicit expressions, this becomes
    \begin{multline*}
        \int_0^1 \left( \alpha_1(x) \beta \left( x-\frac{1}{2} \right) + \alpha_2(x) \beta \left( \frac{1}{2}-x \right) - \beta x \alpha_1(x) + \beta x \alpha_2(x) \right) \mathrm{d}x \\
        = \frac{\beta}{2} \int_0^1 \left( \alpha_2(x) - \alpha_1(x) \right) \mathrm{d}x
        = 0 ,
    \end{multline*}
    since $\mu_1([0,1])=\mu_2([0,1])$.
    
    Finally, the fact that $\pi_1^{(\beta)}$ and $\pi_2^{(\beta)}$ both have mass $\frac{1}{2}$ follows from $\rho_1^{(\beta)}+\rho_2^{(\beta)}=1$.
\end{proof}

\begin{rem}\label{rem:equadiff}
The above proof may look like a ``guess-and-check'', but we could also \emph{find} the formulas~\eqref{eqn:formula_maximum_action} from the condition in Proposition~\ref{prop:max_concave}. Indeed, we look for densities $\rho_1$ and $\rho_2$ such that the right-hand side of~\eqref{eqn:derivative_action} vanishes for any $\alpha_1, \alpha_2$ with $\int \alpha_1=\int \alpha_2$.
In particular, the case $\alpha_2=0$ implies that one should have $\int \alpha_1 \left( \log \frac{\rho_1}{1-\rho_1}-\beta x \right)=0$ whenever $\int \alpha_1=0$. Therefore $\log \frac{\rho_1}{1-\rho_1}-\beta x$ must be equal a.e.~to a constant $c_1$, which means that $\rho_1(x)=\frac{1}{1+\ee^{-\beta x-c_1}}$ a.e., and similarly for $\rho_2$ with another constant $c_2$. Then taking $\alpha_1(x)=\alpha_2(x)=1$ for all $x$ in $[0,1]$ in~\eqref{eqn:derivative_action} shows $c_1=c_2$. Finally, for the minimizer to live in $\mathcal D$, we need to have $\int_0^1 \rho_1=\int_0^1 \rho_2$, which happens only for $c_1=c_2=-\beta/2$.
\end{rem}


Before going to the proof of Theorem~\ref{thm:phase_transition_321}, 
let us use the minimisation result that we have just obtained to study the partition function of the model.

\begin{proof}[Proof of the estimate~\eqref{eq:partition321}]
	By the first estimate in Lemma~\ref{lem:Varadhan}, Corollary~\ref{cor:LDP_Mallows_321} and Proposition~\ref{prop:max_action}, we have that
\[ \lim_{n \rightarrow \infty} \frac{1}{n} \log \frac{Z_n^\beta}{Z_n^0} =- \inf A_\beta =  -A_\beta(\pi_1^{(\beta)}, \pi_2^{(\beta)}).\]
	Recalling that $\lim_{n \rightarrow \infty} \frac{1}{n} \log (Z_n^0)=2 \log(2)$ (since $Z_n^0$ is the $n$-th Catalan number $C_n$),
	it suffices to evaluate $A_\beta(\pi_1^{(\beta)}, \pi_2^{(\beta)})$.
The latter can be (almost) explicitly computed: using the change of variable
$y = \ee^{\beta\left(x-\frac{1}{2}\right)}$, we have $\dd y=\beta y \dd x$, and $\rho^{(\beta)}_1=y/(1+y)$, so that
	\[ \int_0^1 \rho^{(\beta)}_1(x) \log\left( \frac{\rho^{(\beta)}_1(x)}{1-\rho^{(\beta)}_1(x)} \right) \mathrm{d}x =\beta^{-1} \int_{\ee^{-\beta/2}}^{\ee^{\beta/2}}  \frac{\log y }{1+y} \dd y,\]
	while
	\[ \int_0^1 \log\left( {1-\rho^{(\beta)}_1(x)} \right) \mathrm{d}x = - \beta^{-1} \int_{\ee^{-\beta/2}}^{\ee^{\beta/2}} \log(1+y) \frac{\dd y}{y}.\]
	By symmetry the integral with $\rho^{(\beta)}_2$ have the same values and we get
	\begin{equation}\label{eq:Hopt} H(\pi_1^{(\beta)}, \pi_2^{(\beta)}) = \frac{2}{\beta} \int_{\ee^{-\beta/2}}^{\ee^{\beta/2}} \left( \frac{\log y }{1+y} -  \frac{\log (1+y) }{y}\right) \dd y + 2 \log 2.
	\end{equation}
	We now compute, using that $\rho_1^{(\beta)}$ and $\rho_2^{(\beta)}$ both have integral $1/2$,
	\begin{align*}
    \beta \int_0^1 x\left(\rho_1^{(\beta)}(x) - \rho_2^{(\beta)}(x)\right) \dd x 
    &= \beta \int_0^1 \left( x-\frac{1}{2} \right) \left(\rho_1^{(\beta)}(x) - \rho_2^{(\beta)}(x)\right) \dd x 
	\\&= \int_{\ee^{-\beta/2}}^{\ee^{\beta/2}}
	\big( \tfrac1{\beta}\log(y) \big) \frac{y-1}{y+1} \frac{\dd y}{y} \,.
    \end{align*}
	Now, rewriting $(y-1)/(y+1)$ as $1-2/(y+1)$ and observing that
	\[\int_{\ee^{-\beta/2}}^{\ee^{\beta/2}} \frac{\log(y)}{y} \dd y =\tfrac12 (\log(\ee^{\beta/2})^2 - \log(\ee^{-\beta/2})^2) =0,\]
	we get
	\begin{multline*}
	\beta \int_0^1 x\left(\rho_1^{(\beta)}(x) - \rho_2^{(\beta)}(x)\right) \dd x =
	-2 \beta^{-1} \int_{\ee^{-\beta/2}}^{\ee^{\beta/2}} \frac{\log(y)}{y(y+1)} \dd y\\
	=
	2 \beta^{-1} \int_{\ee^{-\beta/2}}^{\ee^{\beta/2}} \log(y) (\tfrac1{y+1}-\tfrac1y) \dd y=
	2 \beta^{-1} \int_{\ee^{-\beta/2}}^{\ee^{\beta/2}} \frac{\log(y)}{y+1} \dd y.
	\end{multline*}
	Removing this quantity
	from \eqref{eq:Hopt} leads to the asymptotic formula anounced in Equation~\eqref{eq:partition321}.
\end{proof}

We can now conclude the proof of Theorem~\ref{thm:phase_transition_321}. 


\begin{proof}[Proof of Theorem~\ref{thm:phase_transition_321}]
    We recall from Corollary~\ref{cor:LDP_Mallows_321} that $(\pi_1^{[n,\beta]}, \pi_2^{[n,\beta]})$ satisfies an LDP at speed $n$ with rate function $A_{\beta}(\pi_1,\pi_2)- \inf(A_{\beta})$.
    In particular, since $A_{\beta}$ has a unique minimizer $(\pi_1^{(\beta)}, \pi_2^{(\beta)})$ by Proposition~\ref{prop:max_action}, the sequence $(\pi_1^{[n,\beta]}, \pi_2^{[n,\beta]})$ converges a.s.~to $(\pi_1^{(\beta)}, \pi_2^{(\beta)})$.
    Since the mapping $\Psi : \mathcal D \to \PAvTDU$ is continuous by Proposition~\ref{prop:321-avoiding-permutons}, it follows that the sequence $\tau_n^{\beta} = \Psi\left(\pi_1^{[n,\beta]}, \pi_2^{[n,\beta]} \right)$ converges a.s.~to the permuton $\Psi\left( \pi_1^{(\beta)}, \pi_2^{(\beta)} \right)$.
    
    To conclude the proof, we only need to check that $\Psi\left( \pi_1^{(\beta)}, \pi_2^{(\beta)} \right)$ is indeed the permuton described by Theorem~\ref{thm:phase_transition_321}. We will drop the dependency in $\beta$ for ease of notation. We denote by $G_1 = G_{2\pi_1}$ and $G_2=G_{2\pi_2} : \left[0,1 \right] \to [0,1]$ the respective quantile functions of the probability measures $2\pi_1$ and $2\pi_2$ as defined in~\eqref{eqn:defn_quantile_function}. Note that $\pi_1, \pi_2$ have positive densities, so $G_1$ and $G_2$ are both increasing bijections. We can write
    \begin{align}
    	\pi_1 \nearrow \pi_2 &= \frac{1}{2} \left( G_1, G_2 \right)_{\#}  \Leb_{[0,1]} \nonumber \\
    	&= \frac{1}{2} \left( X, G_2 \circ G_1^{-1}(X) \right)_{\#} \left( (G_1)_{\#} \Leb_{[0,1]}\right) \nonumber \\
    	&=  \left( X, G_2 \circ G_1^{-1}(X) \right)_{\#} \pi_1. \label{eqn:identification_permuton}
    \end{align}
    We now write $f=G_2 \circ G_1^{-1}$ and check that this coincides with the function $f_{\beta}$ appearing in Theorem~\ref{thm:phase_transition_321}. By definition of quantile functions, for any $u \in [0,1]$, we have
    \[ 2\pi_1 \left( \left[ 0,G_1(u)\right] \right) = u = 2\pi_2 \left( \left[ 0,G_2(u)\right] \right), \]
    so $\pi_1([0,x]) = \pi_2([0,f(x)])$. Since the density $\rho_2$ is continuous and positive, this identity implies that $f$ is $\mathcal C^1$ as well.
    Differentiating yields $\rho_1(x) = f'(x) \rho_2(f(x))$, and therefore:
    \begin{equation*}
    f'(x) = \frac{\rho_1(x)}{\rho_2(f(x))} = \frac{1 + \ee^{\beta(f(x)-1/2)}}{1 + \ee^{\beta(1/2-x)}} \,.
    \end{equation*}
    Substituting $g := \ee^{\beta(f-1/2)}$, we find that
    \begin{equation*}
    g' = q g (1+g)
    \end{equation*}
    where $q(x) := \frac{\beta}{1 + \ee^{\beta(1/2-x)}}$.
    This is a simple Riccati equation, which can be solved\footnote{To solve $g'=qg(1+g)$, consider $u := a\ee^{\beta x}+b$ for abitrary $a\ne0$ and $b\in\R$.
    The latter solves $u'' = R u'$ where $R := q + \frac{q'}{q} = \beta$.
    It can then be checked that $g := \frac{-u'}{qu}$ solves the desired equation, and finally, $c$ is defined as $b/a$.} by
    \begin{equation*}
    g(x) = - \frac{\ee^{\beta x} + \ee^{\beta/2}}{\ee^{\beta x} + c}
    \end{equation*}
    for $c\in\R$ (the only other solution is $g=0$).
    Using $g(0) = \ee^{-\beta/2}$ yields $c = -(1+\ee^{\beta/2}+\ee^\beta)$, and the identification of $f$ follows. 
    Combined with~\eqref{eqn:identification_permuton} and the same computation for $(\Leb_{[0,1]}-\pi_1) \nearrow (\Leb_{[0,1]}-\pi_2)$, this shows that $\Psi(\pi_1,\pi_2)$ is indeed as described by Theorem~\ref{thm:phase_transition_321}.
\end{proof}

\part{Avoiding 231}
\label{part:231}

We now consider permutations and permutons avoiding $231$. 
Again, in this part we will omit the superscript $231$ in all the notations.

\section{231-avoiding permutations and permutons}
\label{sec:structure_231}

\subsection{A classical bijection for 231-avoiding permutations}
\label{ssec:bij231}

If $\sigma$ avoids $231$, let 
\[
    F_\sigma : x\in[0,n) \mapsto \min\left\{ \sigma(c) : c>x \right\} - 1 \,.
\]
This defines a càdlàg, piecewise constant, nondecreasing function\footnote{For readers familiar with the notion of right-to-left minima of a permutation, $F_\sigma$ is the largest integer-valued function whose graph stays below the right-to-left minima of $\sigma$.
In particular, it is completely determined by, and completely determines, the positions and values of the right-to-left minima of $\sigma$.}.
We can let $F_\sigma(n) := n$, so that $F_\sigma$ starts at $0$ and ends at $n$.
Also define $\Phi_\sigma$ as the Dyck path of size $2n$ obtained by rotating $F_\sigma$ by $-\pi/4$, flipping it, and scaling it up by $\sqrt 2$.
See Figure~\ref{fig: example_231} for an illustration.
The following lemma can e.g.~be found in \cite[Proof of Lemma 4.3]{BonaLivre}:
\begin{lem}
The application $\sigma \mapsto \Phi_\sigma$ is one-to-one, from $231$-avoiding permutations to Dyck paths of length $2n$.
\end{lem}
In particular, if $\sigma_n$ is a uniformly random $231$-avoiding permutation with size $n$ then $\Phi_{\sigma_n}$ is a uniformly random Dyck path of length $2n$.
Finally, we define the normalized RLM curve of $\sigma$ as $f_\sigma := \frac1n F_\sigma(n\,\cdot)$ 
and its normalized excursion as $\varphi_\sigma := \frac{1}{2n} \Phi_\sigma(2n\,\cdot)$ on $[0,1]$.

\begin{figure}
    \centering
    \includegraphics{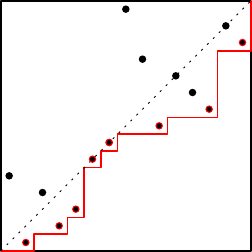}
    \qquad
    \includegraphics{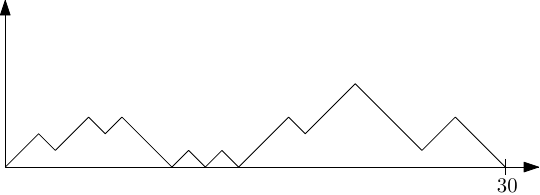}
    \caption{Left: a $231$-avoiding permutation $\sigma$ of size $15$. 
    Its set of RL minima, and the associated curve $F_\sigma$, are shown in red.
    Right: the Dyck path $\Phi_\sigma$ of length $30$ obtained by rotating and flipping the RLM curve.
    }
    \label{fig: example_231}
\end{figure}

\begin{prop}\label{prop: inv of 231 avoiding permutation}
    If $\sigma$ is a $231$-avoiding permutation then
    \[\inv(\sigma) = \binom{n}{2} - \sum_{x=0}^{n-1} F_\sigma(x).\]
    Therefore,
    \[ \inv(\sigma)
    = \binom{n}{2} - n^2 \int_0^1 f_\sigma(x) \mathrm{d}x 
    = \frac{-n}{2} + 2 n^2 \int_0^1 \varphi_\sigma(x) \mathrm{d}x  .\]
\end{prop}

\begin{proof}
 Let $\noninv(\sigma) := \card{ (x,y)\in[n]^2 : x<y , \sigma(x)<\sigma(y) }$ be the number of noninversions of $\sigma$.
    Since $\inv(\sigma) + \noninv(\sigma) = \binom{n}{2}$, we may as well compute $\noninv(\sigma)$.
    For this, we enumerate noninversions by their north-east point.
    Let $c\in[n]$ be fixed.
    We claim that 
    \begin{equation}\label{eq: noninversion set of 231 avoiding}
        \left\{ x\in[n] : x<c , \sigma(x)<\sigma(c) \right\}
        = \left\{ \sigma^{-1}(y) : y \le F_\sigma(c-1) \right\} \,,
    \end{equation}
    and thus
    \begin{equation*}\label{eq: noninversion count of 231 avoiding}
        \card{ x\in[n] : x<c , \sigma(x)<\sigma(c) }
        = F_\sigma(c-1) \,,
    \end{equation*}
    so that $\noninv(\sigma) =  \sum_{x=0}^{n-1} F_\sigma(x).$

Indeed, for $c\in [n],$ by definition of $F_\sigma,$ for any $y \le    F_\sigma(c-1)$, we have $\sigma^{-1}(y)<c $ since there is no point under the RLM curve.

Moreover, by construction $\sigma(c) \ge  F_\sigma(c-1) +1.$ If $ \sigma(c) =  F_\sigma(c-1) +1,$ then the previous remark directly gives \eqref{eq: noninversion set of 231 avoiding}. Now, if  $ \sigma(c) > F_\sigma(c-1) +1,$ then there exists $c'>c$ such that  $ \sigma(c') = F_\sigma(c-1) +1$ (intuitively, $c'$ is the first RL minimum at the right of $c$). Let $ F_\sigma(c-1)< y< \sigma(c).$
If $\sigma^{-1}(y) < c,$ then the permutation $\sigma$ is not $231$-avoiding because we have $\sigma^{-1}(y) < c < c'$ and
$\sigma(c')< y < \sigma(c).$ Therefore, \eqref{eq: noninversion set of 231 avoiding} holds in both cases and this concludes the proof of the first formula.

    Finally, since the non-normalized RLM curve $F_\sigma$ is locally constant and jumps at integer positions,  
    $\sum_{x=0}^{n-1} F_\sigma(x)=\int_0^n F_\sigma(x) \mathrm{d}x$.
    Using $\int_0^1 f_\sigma(x) \mathrm{d}x = \frac{1}{n^2} \int_0^n F_\sigma(x) \mathrm{d}x$ and $\frac12 = \int_0^1 f_\sigma(x) \mathrm{d}x + 2\int_0^1 \varphi_\sigma(x) \mathrm{d}x$, the rest of the Proposition follows.
\end{proof}

\subsection{231-avoiding permutons}

We write $\PAvDTU$ for the set of $231$-avoiding permutons.
Given $\mu \in \PAvDTU$, we define its RLM curve as
\[
    f_\mu :x\in[0,1] \mapsto \max \left\{ y\in[0,1] : \mu\left([x,1]\times[0,y]\right) = 0 \right\} \,.
\]
The function $f_\mu$ is càdlàg, nondecreasing and satisfies $f_\mu(1)=1$ and $f_\mu(x) \le x$
for $x$ in $[0,1]$ (using the fact that permutons have uniform marginals).


Conversely, let $f$ be a nondecreasing càdlàg function from $[0,1]$ to $[0,1]$
satisfying $f(x) \le x$ and $f(1)=1$.
We let $\mathcal{F}$ denote the space of such functions.
We want to construct a $231$-avoiding permuton $\mu_f$ whose RLM-curve is $f$.
The standard way to proceed for $231$-avoiding permutations is a recursive construction from right to left, which is hard to transpose to continuous objects.
Instead, we construct $\mu_f$ by specifying the area of lower-right rectangles $[x,1] \times [0,y]$.

For this, we let $H_f := \{(x,y)\in[0,1]^2 : y \ge f(x)\}$ be the region above the graph of $f$, $G_f=\{(x,y) : f(x^-) \le y \le f(x)\}$ be the completed graph of $f$, and for $(x,y)$ in $[0,1]^2$ we let $G_f(x,y)=G_f \cap ([x,1] \times [0,y])$ be the part of the graph which lies below and to the right of $(x,y)$.
Finally, we define $P_f:[0,1]^2 \to [0,1]$ as follows:
\begin{equation}\label{eq: cumulative formula for 231-avoiding permuton}
    P_f(x,y)=
    \begin{cases} 
        y-x + \min_{(x',y') \in G_f(x,y)} (x'-y') &\text{ if } (x,y)\in H_f ;\\
        0 & \text{ if }y <f(x).
    \end{cases}
\end{equation}
Note that $P_f(x,y) = \min_{(x',y') \in G_f} ||(x,y) - (x',y')||_1$  if $(x,y)\in H_f$.
In particular, $P_f(x,y) = 0$ if $(x,y)\in G_f$, and $P_f$ is continuous.
See Figure~\ref{fig: cumulative formula for 231-avoiding permuton} for an illustration of \eqref{eq: cumulative formula for 231-avoiding permuton}.

\begin{figure}
    \centering
    \includegraphics{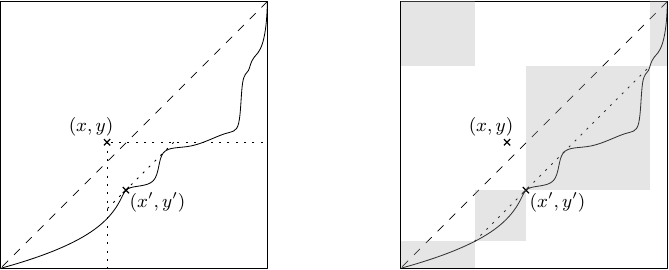}
    \caption{Illustration of \eqref{eq: cumulative formula for 231-avoiding permuton}. 
    Left: 
    the point $(x',y')$ is one of the closest to the diagonal of $[0,1]^2$ in $G_f(x,y)$.
    Right:
    It can actually be shown that the permuton with RLM-curve $f$ has support contained in the shaded area (see the proof of Lemma~\ref{lem:231_Uniqueness}), explaining the formula \eqref{eq: cumulative formula for 231-avoiding permuton}.}
    \label{fig: cumulative formula for 231-avoiding permuton}
\end{figure}

\begin{lem}\label{lem:construction_muf}
Let $f\in\mathcal{F}$.
There exists a unique $231$-avoiding permuton $\mu_f$ such that
\begin{equation}\label{eq:def_muf}
    \text{for all }(x,y)\in[0,1]^2,\quad
    \mu_f([x,1] \times [0,y])=P_f(x,y).
\end{equation}
Moreover, the RLM-curve of $\mu_f$ is $f$.
\end{lem}

\begin{proof}
The uniqueness is straightforward. 
It remains to construct a measure $\mu_f$ satisfying \eqref{eq:def_muf} and to verify that it is a permuton, that it is $231$-avoiding, and that its RLM-curve is $f$.

We first check that for any pair $(x_1,y_1)$, $(x_2,y_2)$ in $H_f$ with $x_1<x_2$ and $y_1 >y_2$, we have
\begin{equation}
    \label{eq:Condition_Bivariate_Cumulative_Repartition_Function}
    P_f(x_1,y_1) - P_f(x_1,y_2) - P_f(x_2,y_1) + P_f(x_2,y_2) \ge 0.
\end{equation} 
It is equivalent to check that
\[\min_{(x',y') \in G_f(x_1,y_1)} \! (x'-y')
-\min_{(x',y') \in G_f(x_1,y_2)} \! (x'-y')
-\min_{(x',y') \in G_f(x_2,y_1)} \! (x'-y')
+\min_{(x',y') \in G_f(x_2,y_2)} \! (x'-y') \ge 0.\]
But the above follows from Lemma~\ref{lem:inf} since $G_f(x_1,y_2) \cap G_f(x_2,y_1) =G_f(x_2,y_2)$ and $G_f(x_1,y_2) \cup G_f(x_2,y_1) =G_f(x_1,y_1)$.
For the latter equality, it is important to note that $(x_2,y_2) \in H_f$.
Hence~\eqref{eq:Condition_Bivariate_Cumulative_Repartition_Function} is proved. 

We now define a pre-measure $\mu_f$ on finite unions of semi-open rectangles included in $H_f$ by
\[\mu_f([x_1,x_2) \times (y_2,y_1]) = 
P_f(x_1,y_1) - P_f(x_1,y_2) - P_f(x_2,y_1) + P_f(x_2,y_2).\]
By Carathéodory's extension theorem, we can extend $\mu_f$ to a Borel measure on the set $H_f$ that we still denote $\mu_f$.
This measure $\mu_f$ can be further extended to a measure on $[0,1]^2$ by putting no mass on the complement set  $\{(x,y) \in [0,1]^2: y < f(x)\}$.
Note that $\mu_f$ indeed satisfies \eqref{eq:def_muf}. Let us prove that it is a permuton: for $y \in [0,1]$, we have
\[\mu_f([0,1]\times [0,y])=P_f(0,y)=y - \min_{(x',y') \in G_f(0,y)} (x'-y').\]
But since $f(0)=0$ and $f(x) \le x$, the minimum is reached for $x'=y'=0$, and we get $\mu_f([0,1]\times [0,y])=y$.
Similarly, $\mu_f([x,1]\times [0,1])=1-x$ (in this case, the minimum is reached for $x'=y'=1$), proving that $\mu_f$ is a permuton.

The next step is to prove that $\mu_f$ is $231$-avoiding. The proof will be illustrated by Figure~\ref{fig:proof_muf_avoids_231}.
Assume for the sake of contradiction that we have three points $(x_1,y_1)$, $(x_2,y_2)$ and $(x_3,y_3)$ in the support of $\mu_f$ such that $x_1<x_2<x_3$ and $y_3<y_1<y_2$.
Since $(x_3,y_3)$ is both the lowermost and rightmost point of the three,
the points $(x_1,y_1)$, $(x_2,y_2)$ cannot lie on the curve $G_f$,
and we necessarily have $y_i > f(x_i)$ for $i \in \{1,2\}$. Moreover, since $(x_3, y_3)$ lies weakly above $G_f$, we can write
\begin{equation}\label{eqn:x2_and_y1}
x_2<x_3 \leq f^{-1}(y_3) \leq f^{-1}(y_1).
\end{equation}
For small $\eps>0$, we have $\mu_f([x_1-\eps,x_1+\eps) \times (y_1-\eps,y_1+\eps])>0$.
Looking at the strict inequality case in \eqref{eq:Condition_Bivariate_Cumulative_Repartition_Function} and in Lemma~\ref{lem:inf}, this implies that the minimum of $x'-y'$ over $G_f(x_1+\eps,y_1+\eps)$ is reached only on the horizontal band
\[
[x_1+\eps,1] \times [y_1-\eps,y_1+\eps].
\]
Let $m_1$ be a point in this band where the minimum is attained, as on Figure~\ref{fig:proof_muf_avoids_231}. Similarly, the minimum of $x'-y'$ over $G_f(x_2-\eps,y_2-\eps)$ is reached only on the vertical band
\[ [x_2-\eps,x_2+\eps] \times [0,y_2-\eps]\]
and it is reached at a point $m_2$.
By~\eqref{eqn:x2_and_y1}, we have $x_1<x_2<f^{-1}(y_1)$ so for $\eps$ small enough,
the intersection of $G_f$ with the band $[x_1+\eps,1] \times [y_1-\eps,y_1+\eps]$ is included in $G_f(x_2-\eps,y_2-\eps)$, so $m_1 \in G_f(x_2-\eps,y_2-\eps)$. But the definition of $m_2$ implies that $m_2$ is strictly closer than $m_1$ to the main diagonal. Vice-versa, we also have $m_2 \in G_f(x_1+\eps,y_1+\eps)$, so by definition $m_1$ is strictly closer than $m_2$ to the main diagonal, and we get a contradiction (see Figure~\ref{fig:proof_muf_avoids_231}). This proves that $\mu_f$ is $231$-avoiding.

Finally, we verify that the RLM-curve of $\mu_f$ is $f$. Let us call $g$ this RLM-curve.
That $g(x) \ge f(x)$ is trivial since for any $y < f(x)$, one has $\mu_f([x,1] \times [0,y])=0$. Let us prove the reverse inequality.
Fix $x$ in $[0,1]$, and consider $y>f(x)$. Since $f$ is càdlàg, there exists $\eps>0$ such that for $x' \in [x,x+\eps]$, we have $f(x') \le y-\eps$. This implies that all points $(x',y')$ on $G_f(x,y)$ are at $L^1$-distance at least $\eps$ from $(x,y)$ 
(either they satisfy $x' \ge x+\eps$ or $y'\le y-\eps$),
so that $\mu_f([x,1] \times [0,y]) \ge \eps >0$, so $g(x) \le y$. We conclude by letting $y \to f(x)$.
\end{proof}
\begin{figure}
    \centering
    \includegraphics{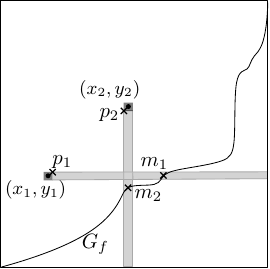} 
    \caption{Illustration of the proof of Lemma~\ref{lem:construction_muf}. Left: we consider two hypothetical points $(x_1,y_1)$ and $(x_2,y_2)$ with $y_1 <y_2$ and $x_1<x_2<f^{-1}(y_1)$.
    The minimum of $x'-y'$ on the curve $G_f$ below and on the right of $p_1:=(x_1+\eps,y_1+\eps)$ is reached somewhere
    in the horizontal light gray band, namely in $m_1$.
    On the other hand, the minimum of $x'-y'$ on the curve $G_f$ below and on the right of $p_2:=(x_2-\eps,y_2-\eps)$
    is reached in the vertical light gray band, resp.~$m_2$. This is impossible, since $m_1$ would need to be at the same time closer and further from the main diagonal than $m_2$.
	}
    \label{fig:proof_muf_avoids_231}
\end{figure}

Our next goal is to prove that the map $\mu \mapsto f_\mu$ is one-to-one on the space of $231$-avoiding permutons.

\begin{lem}\label{lem:231_Uniqueness}
Let $f\in\mathcal{F}$.
Then, the above construted permuton $\mu_f$ is the unique $231$-avoiding permuton
whose RLM-curve is equal to $f$.
\end{lem}

\begin{proof}
    Let $\nu$ be a $231$-avoiding permuton whose RLM-curve is equal to $f$.
    A point $(x,y)$ in $[0,1]^2$ with $y \ge f(x)$ will be called an excursion corner of $f$ if the minimum of $x'-y'$ on $G_f(x,y)$ is attained both at $(x,f(x))$ and at $(f^{-1}(y),y)$.
    We let $S_{x,y}$ be the corresponding square $[x,f^{-1}(y)] \times [f(x),y]$.
    We claim that necessarily $\nu(S_{x,y})=f^{-1}(y)-x=y-f(x)$, i.e.~that there is no mass on the left, on the right, below or above this square.
    
    That there is no mass below and on the right of $S_{x,y}$ is trivial since $f$ is the RLM-curve of $\nu$.
    Assume that there is some mass on the left, then we can find a point $(x_1,y_1)$ in the support of $\nu$, verifying $x_1<x$ and $f(x)<y_1<y$ and $\nu([0,x] \times [f(x),y_1])>0$.
    Consider the square $S=[x,x+y_1-f(x)] \times [f(x),y_1]$; see Figure~\ref{fig:231_Uniqueness}, left. Since $\nu$ is a permuton with $\nu([0,x] \times [f(x),y_1])>0$, we have
    \[ \nu(S) < \nu( [0,1] \times [f(x),y_1] ) =y_1-f(x).\]
    Thus there is some mass of $\nu$ above $S$ (since $f$ is the RLM curve of $\nu$, there cannot be any mass of $\nu$ below $S$),
    i.e.~there exists a point $(x_2,y_2)$ in the support of $\nu$ with $x<x_2<x+y_1-f(x)$ and $y_2>y_1$.
    Since $(x,y)$ is an excursion corner, the restriction of the curve $G_f$ to $[x,x+y_1-f(x)]$ lies inside $S$ and we can find $(x_3,y_3)$ on 
    $G_f \cap (x_2, x+y_1-f(x)] \times [f(x), y_1)$.
    The next step is to let $x'_3$ be maximal such that $(x'_3,y_3)$ is on $G_f$;
    note that $(x'_3,y_3)$ might be outside $S$, but this ensures that $(x'_3,y_3)$ is in the support of $\nu$ (while $(x_3,y_3)$ might not be).
    Finally, we have constructed three points $(x_1,y_1)$, $(x_2,y_2)$ and $(x'_3,y_3)$ which are in the support of $\nu$ and satisfy $x_1<x<x_2<x_3 \le x'_3$ and $y_3<y_1<y_2$, i.e.~form a $231$ pattern.
    This is in contradiction with the assumption that $\nu$ is $231$-avoiding.
    We have proved that, for any excursion corner $(x,y)$ there is no mass on the left of $S_{x,y}$.
    A similar argument ensures that there is no mass above $S_{x,y}$ either.
    
    We now consider any point $(x,y)\in(0,1)^2$ (not necessarily an excursion corner) weakly above the RLM-curve $f$ and we let $(x',y')$ minimizing the quantity $x'-y'$
    on the set $G_f(x,y)$. We denote by $(x_l,y_l)$ the rightmost point on $G_f$ for which $x_l<x$ and $x_l-y_l=x'-y'$.
    By construction $(x_l,y')$ is an excursion corner of $f$. 
    See Figure~\ref{fig:231_Uniqueness}, right.
    Similarly, let $(x_r,y_r)$ be the leftmost point on $G_f$ satisfying $y_r \geq y$ and $x_r-y_r=x'-y'$.
    Again $(x',y_r)$ is an excursion corner of $f$. Using that $\nu$ has no mass below $G_f$, nor on the left of or above excursion squares, we can write
    \[ \nu([x,1] \times [0,y]) = \nu([x,x'] \times [y_l,y']) + \nu([x',x_r] \times [y',y])=x'-x+y-y'.\]
    Hence $\nu$ coincides with the permuton $\mu_f$ introduced in Lemma~\ref{lem:construction_muf}.
    \end{proof}
    
\begin{figure}
    \[
    \includegraphics{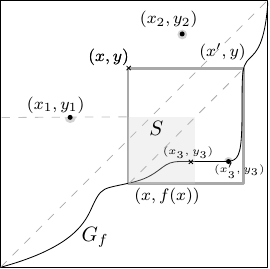} \qquad \includegraphics{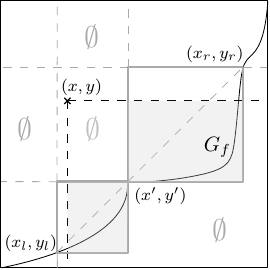}\]
    \caption{Illustration of the proof of Lemma~\ref{lem:231_Uniqueness}. The black curve is the graph $G_f$ of $f$, and gray fat squares represent excursions of $f$. The black disks with some gray zones around are points of the support of $\nu$ with some zone having a positive mass around them. The crosses are points that are not necessarily in the support.  {\em Left:} we assume by contradiction that there is some mass below a point $(x_1,y_1)$ 
    on the left of the excursion square. Then there must be some mass around $(x_2,y_2)$
    which is above the light-gray square $S$ (notation from the proof),
    and around a point $(x'_3,y_3)$ on $G_f$, which creates the forbidden pattern 231.
    {\em Right:} We want to compute the $\nu([x,1] \times [0,y])$.  Empty set symbols indicate zones where $\nu$ does not have any mass.
    The mass  $\nu([x,1] \times [0,y])$ is obtained by summing the contributions
    of the two light gray rectangles.}
    \label{fig:231_Uniqueness}
\end{figure}

We end this section with a more satisfying description of the permuton $\mu_f$ in a special case:
if the RLM curve is strictly convex and symmetric, then the support of $\mu_f$ is the union of its graph with a piece of the antidiagonal, with explicit densities.
This special case will be the case of interest for us later.

\begin{prop}\label{prop: support and density of 231 permuton}
    Let $f\in\mathcal{F}$, and let $\mu_f$ be the unique permuton with RLM curve $f$.
    Assume the following:
    \begin{enumerate}
        \item $f$ is increasing and continuous on $[0,1]$. 
        In particular, there is a unique $x^*\in(0,1)$ such that $f(x^*) = 1-x^*$.
        \item The graph $G_f$ of $f$ is invariant under the symmetry $(x,y) \mapsto (1-y, 1-x)$, that is: 
        \[ \text{for all } x\in[0,1],\quad f(1-f(x)) = 1-x .\]
        \item $f$ is strictly convex.
    \end{enumerate}
    Then $\mu_f$ equals $\mu_\nearrow + \mu_\searrow$ where:
    \begin{itemize}
        \item $\mu_\nearrow$ is the pushforward of the measure with cumulative distribution function (CDF) 
        \[ F_\nearrow : x\in[0,1] \mapsto \left\{
        \begin{array}{ll}
            f(x)
            &\text{ if } x\le x^* \,;\\
            1 + x- 2x^*
            &\text{ if } x\ge x^* \,;
        \end{array}
        \right. \]
        via $x\mapsto(x,f(x))$;
        \item $\mu_\searrow$ is the pushforward of the measure with CDF
        \[ F_\searrow : x\in[0,1] \mapsto \left\{
        \begin{array}{ll}
            x - f(x)
            &\text{ if } x\le x^* \,;\\
            2x^* - 1
            &\text{ if } x\ge x^* \,;
        \end{array}
        \right. \]
        via $x\mapsto(x,1-x)$.
    \end{itemize}
    The support of $\mu_\nearrow$ is the graph $G_f$ of $f$, and the support of $\mu_\searrow$ is the antidiagonal line from $(1,0)$ to $(x^*, f(x^*))$.
\end{prop}

\begin{proof}
    Let $\mu := \mu_\nearrow + \mu_\searrow$ be the measure defined by the proposed formulas.
    First, let us check that these measures are well-defined.
    By (1), the functions $F_\nearrow$ and $F_\searrow$ are continuous and the function $F_\nearrow$ is increasing.
    Then, by (2) and (3), the function $x\mapsto x-f(x)$ is increasing on $[0,x^*]$ and decreasing on $[x^*,1]$, and so $F_\searrow$ is monotonous.
    Thus, these functions are indeed the CDFs of two measures on $[0,1]$.

    The claim about the supports of $\mu_\nearrow$ and $\mu_\searrow$ is a direct consequence of the assumptions made on $f$.
    Note that the support of $\mu_\searrow$ can be described as
    \[ \{ (x,1-x) : x\le x^* \} = \{ (x,1-x) : 1-x\ge f(x^*)\} \,. \]
    In particular, it is immediate from the description of the support of $\mu$ that this measure is $231$-avoiding and that its RLM curve is $f$.

    Now, let us check that $\mu$ is a permuton.
    Since $F_\nearrow(x) + F_\searrow(x) = x$ for all $x\in[0,1]$, the first marginal of $\mu$ is indeed uniform.
    It remains to check that the second marginal is uniform as well.
    Fix $0\le y\le z\le 1$, and let us check that $\mu([0,1]\times[y,z]) = z-y$.
    It suffices to check this in the two cases $z\le f(x^*)$ and $y\ge f(x^*)$.

    First, assume $z\le f(x^*)$.
    The support of $\mu_\searrow$ does not intersect the band $[0,1]\times[y,z]$.
    Also, $f^{-1}(z)\le x^*$.
    Hence:
    \[
        \mu([0,1]\times[y,z]) 
        = F_\nearrow\left( f^{-1}(z) \right) - F_\nearrow\left( f^{-1}(y) \right)
        = z-y \,.
    \]
    as desired.

    Now, assume $y\ge f(x^*)$.
    Since $f^{-1}(y)\ge x^*$, we have that:
    \begin{multline*}
        \mu([0,1]\times[y,z]) 
        = F_\searrow(1-y) - F_\searrow(1-z)
        + F_\nearrow\left( f^{-1}(z) \right) - F_\nearrow\left( f^{-1}(y) \right)
        \\= ( 1-y-f(1-y) ) - ( 1-z-f(1-z) ) + ( f^{-1}(z) - f^{-1}(y) ) \,.
    \end{multline*}
    By the symmetry of $f$, it holds that $f^{-1}(x) = 1 - f(1-x)$ for all $x\in[0,1]$.
    Therefore, $\mu([0,1]\times[y,z]) = z-y$ as desired.

    In conclusion, we have shown that $\mu$ is a $231$-avoiding permuton with RLM curve $f$.
    By Lemma~\ref{lem:231_Uniqueness}, this shows that $\mu_f = \mu$.
\end{proof}

\subsection{Bicontinuity of the RLM curve parametrization}\label{subsec:231_continuity}
Recall that $\mathcal{F}$ is the space of nondecreasing càdlàg functions from $[0,1]$ to $[0,1]$ such that $f(1)=1$ and $f(x)\le x$ for all $x\in[0,1]$.
For a permuton $\mu$, its RLM-curve $f_\mu$ lives in $\mathcal F$.
To discuss continuity properties (and for large deviation principles later), it is convenient to introduce also $\varphi_\mu$ as the $1$-Lipschitz function on $[0,1]$ whose graph is obtained by rotating $G_{f_\mu}$ by $-\pi/4$, flipping it, and scaling it down by $\sqrt2$ (as on \Cref{fig: example_231,fig:parametrization}).
The function $\varphi_\mu$ is $1$-Lipschitz, nonnegative, and satisfies $\varphi_\mu(0)=\varphi_\mu(1)=0$.
We recall that we denote by $\mathcal E$ the space of such functions, 
endowed with the uniform norm.

Finally, we will also consider the (completed) graph $G_f$ of $f$ as living in the space of compact subsets of $[0,1]^2$, endowed with the Hausdorff distance $d_H$.
Namely for compact sets $G_1$ and $G_2$, 
\[d_H(G_1,G_2) = \inf \{\eps>0:\ G_1 \subseteq G_2^\eps \mbox{ and } G_2 \subseteq G_1^\eps\},\]
where for a compact set $G$ and $\eps>0$, the $\eps$-halo $G^\eps$ of $G$ is the set of points at distance at most $\eps$ from $G$.
This notion depends on an underlying metric on the ambient space, here $[0,1]^2$, and we will use the supremum norm on $[0,1]^2$ in the sequel.

\begin{lem}\label{lem: equiv cv f and G_f and varphi}
    Let $f$ and $f_n$, $n\ge1$, be functions in $\mathcal{F}$.
    Let $\varphi$ and $\varphi_n$, $n\ge1$, be the corresponding functions in $\mathcal{E}$.
    The following assertions are equivalent:
    \begin{enumerate}[label=(\roman*)]
        \item $f_n(x) \to f(x)$ for all continuity points $x$ of $f$;
        \item $G_{f_n} \to G_f$ for the Hausdorff distance;
        \item $\varphi_n \to \varphi$ for the uniform distance.
    \end{enumerate}
\end{lem}

\begin{proof}
    \underline{(i)$\implies$(ii).}
    Assume that $f_n \to f$ pointwise at continuity points of $f$.
    Fix $\eps>0$.
    Since $f$ is nondecreasing, almost every point is a continuity point and we can construct a sequence $0=x_0 < x_1 < \dots < x_{\ell-1} < x_\ell=1$ of continuity points (except maybe $x_\ell$, but this does not matter in the following) which is a subdivision of $[0,1]$ with steps bounded by $\eps$.
    For each $k<\ell$, $G_f \cap \left( (x_k,x_{k+1}]\times[0,1] \right)$ is a continuous nondecreasing curve from $(x_k, f(x_k))$ to $(x_{k+1}, f(x_{k+1}))$, and thus its $\eps$-halo covers the rectangle $[x_k, x_{k+1}] \times [f(x_k)-\eps, f(x_{k+1})+\eps]$.
    
    For large enough $n$, we have $\left| f_n(x_k) - f(x_k) \right| < \eps$ for all $k\in[0,\ell]$.
    Therefore $G_{f_n} \cap \left( (x_k,x_{k+1}]\times[0,1] \right)$ is contained in the rectangle $[x_k, x_{k+1}] \times [f(x_k)-\eps, f(x_{k+1})+\eps]$.
    Since this holds for each $k<\ell$, this proves that for large enough $n$, the graph $G_{f_n}$ is contained within the $\eps$-halo of $G_f$.
    The reverse inclusion is proved with the same reasoning.
    This proves the Hausdorff convergence of completed graphs.
    \medskip
    
    \underline{(ii)$\implies$(iii).}
    Assume that $G_{f_n} \to G_f$ for the Hausdorff distance, and let $G_{\varphi}$ and $G_{\varphi_n}$, $n\ge1$, be their rotated-flipped-rescaled analogs.
    Obviously, $G_{\varphi_n} \to G_\varphi$ for the Hausdorff distance as well.
    Now let $\eps>0$.
    For large enough $n$, the graph $G_\varphi$ is contained within the $\eps$-halo of $G_{\varphi_n}$.
    Consequently, for any $x$ in $[0,1]$, there exists $x_n\in[0,1]$ such that $|x_n-x|<\eps$ and $|\varphi_n(x_n) - \varphi(x)| < \eps$.
    Since $\varphi_n$ is $1$-Lipschitz, we deduce that $|\varphi_n(x) - \varphi(x)| < 2\eps$.
    This proves that $\varphi_n \to \varphi$ uniformly.
    \medskip
    
    \underline{(iii)$\implies$(i).}
    Assume that $\varphi_n \to \varphi$ uniformly, and let $x\in[0,1]$ be a continuity point of $f$.
    Let $y\in[0,1]$ be such that the point $(y,\varphi(y))$ corresponds to the point $(x,f(x))$ after rotation, flip, and rescaling, see Figure~\ref{fig: topologies f and varphi}.
    Likewise, let $y_n\in[0,1]$ be such that the point $(y_n,\varphi_n(y_n))$ corresponds to the point $(x,f_n(x))$.
    
    Let us consider a subsequence $y_{n_k}$ of $y_n$ which converges to a limit point $y_\infty$.
    Since $\varphi_n \to \varphi$ uniformly, we get that $(y_{n_k}, \varphi_{n_k}(y_{n_k})) \to (y_\infty, \varphi(y_\infty))$.
    Let $(x_\infty, z_\infty) \in G_f$ be the point corresponding to $(y_\infty, \varphi(y_\infty))$, and note that $f(x_\infty^-) \le z_\infty \le f(x_\infty)$.
    By continuity of scaling operations, we have $(x, f_{n_k}(x)) \to (x_\infty, z_\infty)$.
    Therefore $x_\infty = x$, and since $f$ is continuous at $x$, we deduce that $z_\infty=f(x)$.
    In particular, $y_\infty = y$, i.e.~$y$ is the only possible limit point of the sequence $y_n$. This implies $y_n \to y$ so $(y_n, \varphi_n(y_n)) \to (y,\varphi(y))$ and finally $(x,f_n(x))$ converges to $(x,f(x))$.
\end{proof}

\begin{rem}
The equivalence $(G_{\varphi_n} \to G_\varphi)\ \Longleftrightarrow\ (\varphi_n \to \varphi)$ is a special case of a general result
relating convergence of functions and convergence of their graphs, see e.g.~\cite{waterhouse1976graph-convergence}.
\end{rem}

\begin{figure}
    \centering
    \includegraphics{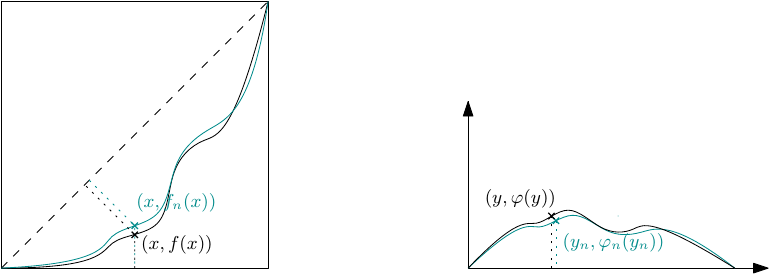}
    \caption{Construction of the sequence $(y_n)$ in the proof of Lemma~\ref{lem: equiv cv f and G_f and varphi}, (iii)$\implies$(i).}
    \label{fig: topologies f and varphi}
\end{figure}

\begin{lem}\label{lem:bicontinuity}
    The correspondence $\mu \in \PAvDTU \mapsto \varphi_\mu \in \mathcal E$ is bicontinuous.
\end{lem}

\begin{proof}
	Since $\mathcal{E}$ is compact, it is sufficient to check the continuity of the reverse mapping. Therefore, let $(\varphi_n)$ be a sequence in $\mathcal{E}$ converging towards $\varphi$, and let $(\mu_n)$ and $\mu$ be the corresponding $231$-avoiding permutons. Let $(G_n)$ and $G$ be the completed graphs of their RLM curves, and let $(H_n)$ and $H$ be the parts of $[0,1]^2$ weakly above $(G_n)$ and $G$.
	Thanks to Lemma~\ref{lem: equiv cv f and G_f and varphi}, it holds that $G_n \to G$ for the Hausdorff distance, and therefore $\overline{[0,1]^2\setminus H_n} \to \overline{[0,1]^2\setminus H}$ as well.
	For any $x,y \in [0,1]$, let $P_n(x,y) = \mu_n\left( [x,1]\times[0,y] \right)$. We recall from~\eqref{eq: cumulative formula for 231-avoiding permuton} that
	\begin{equation}\label{eq: bicontinuity proof - cumulative formula for 231-avoiding permuton}
	P_n(x,y) = \mathrm{dist}_{L^1}\big( (x,y) \,,\, [0,1]^2\setminus H_n \big)
	\end{equation}
	Since the map $\mathrm{dist}_{L^1}\big( (x,y) , \cdot \big)$ is continuous on closed subsets of $[0,1]^2$ for the Hausdorff distance, we get that
	\[
	P_n(x,y) = \mathrm{dist}_{L^1}\big( (x,y) \,,\, \overline{[0,1]^2\setminus H_n} \big)
	\to \mathrm{dist}_{L^1}\left( (x,y) \,,\, \overline{[0,1]^2\setminus H} \right) = P(x,y) \,.
	\]
	Since this holds for all $(x,y)\in[0,1]^2$, this proves that $\mu_n \to \mu$.

\end{proof}

\section{Large deviation principle for 231-avoiding permutations}

The goal of this section is to provide a proof of Proposition \ref{prop:LDP_231_pi_intro}, that is, to establish a large deviation principle for $\varphi_{\sigma_n}$ when $\sigma_n$ is a uniform random $231$-avoiding permutation of size $n$.
As pointed out in the previous section, $\varphi_{\sigma_n}$ is a uniform Dyck path of length $2n.$ 
The proof will be similar to the proof of Proposition \ref{prop:LDP_321_pi_intro} but we again emphasize that, although the underlying discrete objects are the same, the topology in which we work is different in both cases. 
As in Part I, our strategy will consist in comparing our random objects with a simpler model.

%
%
    Let $(\eps_i)_{i \ge 1}$ be i.i.d.~Rademacher random variables, $S_k := \sum_{i=1}^k \eps_i$,
    \[
        Z_n : t \in [0,1] \mapsto \frac{1}{2n} S_{\lfloor 2nt \rfloor},
    \]
    and $W_n$ its continuous polygonal interpolation. We recall that the function $J$ is defined in~\eqref{eqn:defn_J}. We will write $\mathcal{C}=\mathcal{C}([0,1])$, and we recall from~\eqref{eqn:defn_F} the definition of the compact subset $\mathcal{E} \subset \mathcal{C}$. By Mogulskii's theorem (see \cite{mogulskij1976LDP_Random_Walks}, or \cite[Theorem 5.1.2]{dembo1998large-deviations}), $(W_n)_{n \ge 1}$ satisfies an LDP on $\mathcal{C}$ at speed $n$ with good rate function
    \[
        H(\varphi) =
        \begin{cases}
            2 \int_0^1 J(\varphi'(t)) \mathrm{d}t & \text{if $\varphi$ is 1-Lipschitz and } \varphi(0)=0;\\
            \infty & \text{otherwise.}
        \end{cases}
    \]
   The random variables $\varphi_{\sigma_n}$ can be written as $W_n$ conditioned on $\{S_{2n}=0,\ \min_{k \le 2n} S_k \ge 0\}$.

\begin{proof}[Proof of Proposition~\ref{prop:LDP_231_pi_intro}]
For $\varphi \in \mathcal C$ and $\varepsilon >0,$ we denote by
\[ B(\varphi , \varepsilon ) := \{ f\in  \mathcal C : \| \varphi-f\|_\infty \le \varepsilon\}\]
the ball of radius $\varepsilon$ for the uniform norm.
By construction, for $n\ge 1,$ both $\varphi_{\sigma_n}$ and $W_n$ belong to the ball $B(0,1) \subset \mathcal C,$ which is compact, so it suffices to establish a weak LDP in $B(0,1).$ As   $\varphi_{\sigma_n} \in \mathcal E,$ which is closed, Proposition~\ref{prop:LDP_231_pi_intro} is obtained by restriction.
More precisely, we have to show the following upper and lower bounds:
\[ \lim_{\varepsilon \downarrow 0} \limsup_{n \rightarrow\infty} 
\frac1n \log
\mathbb P(\varphi_{\sigma_n} \in B(\varphi , \varepsilon )) \le - H(\varphi)\]
and

\[ \lim_{\varepsilon \downarrow 0} \liminf_{n \rightarrow\infty} 
\frac1n \log
\mathbb P(\varphi_{\sigma_n} \in B(\varphi , \varepsilon )) \ge - H(\varphi),\]
 for any $\varphi\in \mathcal E.$
Indeed, if $\varphi \notin \mathcal E,$ for $\varepsilon$ small enough and $n$ large enough, the event $\{\varphi_{\sigma_n} \in B(\varphi , \varepsilon )\}$ is empty and the bounds are trivial.

 The upper bound is easy to show.
    For any Borel set $B$ in $\mathcal{C}([0,1])$, counting paths gives
    \begin{align*}
        \mathbb{P}\big[ \varphi_{\sigma_n}\in B \big]
       & = \frac{\#\{\text{Dyck paths of length $2n$ in }B\}}{C_n}\\
        & \le \frac{\#\{\text{walks of length }2n\text{ in }B\}}{C_n}
        = \frac{4^n}{C_n}\, \mathbb{P}\big[ W_n \in B \big],
    \end{align*}
    where we recall that $C_n$ is the Catalan number and $4^n/C_n = \Theta(n^{3/2})$ is subexponential. This is the same counting argument as in the proof of Proposition \ref{prop:LDP_321_pi_intro}.
   Now, if we choose $B=  B(\varphi , \varepsilon ) \cap \mathcal{E}$, which is a closed subset of $\mathcal{C}$, and we use the LDP for $(W_n),$ we get the upper bound.

 We now go to the proof of the lower bound.
    Fix $\varphi \in \mathcal E$ with $H^{231}(\varphi)<\infty$ and $\eps>0$. 
    Note that $\varphi$ is nonnegative and that $\varphi(0) = \varphi(1) = 0.$

    For any possible path $W_n,$ we denote by $(S_k)_{0 \le k \le 2n}$ the corresponding walk. We introduce
    $m_n := -\min_{k \le 2n} S_k \ge 0$ and $a_n := S_{2n}$ and note that $a_n+2m_n \ge 0$.
    If $W_n \in B(\varphi,\varepsilon),$ we have $m_n \le 2\eps n$ and $|a_n| \le 2\eps n$ since $\varphi$ is nonnegative on $[0,1]$ and $\varphi(1) = 0.$
	The scheme of the proof of the lower bound is very similar to the proof of Proposition~\ref{prop:LDP_321_pi_intro}. We define a map $T$ from paths to Dyck paths as follows:
    \begin{itemize}
        \item Flip the first $m_n$ negative steps of the path to $+1$ (this lifts the minimum to $0$).
        \item If $a_n+2m_n>0$, flip the last $\frac{a_n+2m_n}{2}$ positive steps to $-1$; 
        if $a_n+2m_n=0$, do nothing.
    \end{itemize}
    After the first item, the path is nonnegative and ends at height $a_n+2m_n$, and the quantity $\frac{|a_n+2m_n|}{2}$ is an integer since $a_n$ is even.
    The second item then brings the endpoint to $0$ without creating negative heights.
    The resulting path $T(W_n)$ is therefore a Dyck path of length $2n$.
    Moreover, since at most $m_n+\frac{|a_n+2m_n|}{2} \le 5\eps n$ steps are flipped, the sup norm changes by at most $5\eps.$


Let us evaluate the number of preimages of $T(W_n).$ The choice of flipped steps is determined by which $m_n$ negative steps are turned positive and which $\frac{|a_n+2m_n|}{2}$ steps are flipped by the second item.
    As in the proof of Proposition~\ref{prop:LDP_321_pi_intro},  for $\varepsilon$ such  that $5\varepsilon<1/2,$ the number of possibilities is bounded by
    \[
        \sum_{k \le 5\eps n} \binom{2n}{k} \le 5\eps n \binom{2n}{5\eps n}
        \le \frac{ e \sqrt{10\eps n}}{2\pi \sqrt{1-10\eps}} \big( (10\eps)^{10\eps} (1-10\eps)^{1-10\eps} \big)^{-n},
    \]
    which is subexponential in $n$.
    Therefore
\[    \lim_{\varepsilon \downarrow 0} \liminf_{n \rightarrow\infty} 
\frac1n \log
\mathbb P(\varphi_{\sigma_n} \in B(\varphi , 5\varepsilon )) \ge
    \lim_{\varepsilon \downarrow 0} \liminf_{n \rightarrow\infty} 
    \frac1n \log
    \mathbb P(W_n \in B(\varphi , \varepsilon )) \ge - H(\varphi) = - H^{231}(\varphi),\]
   where we used the lower bound of the LDP for $(W_n)$ and the fact that $H(\varphi) = H^{231}(\varphi)$ whenever they are both finite.
\end{proof}

As in Part I, the last step is to apply Lemma \ref{lem:Varadhan} to deduce an LDP for Mallows random permutations. 
We define the action $A_{\beta}$ on $\mathcal E$ by
\[ A_{\beta}(\varphi) = H(\varphi) - 2 \beta \int_0^1  \varphi(t)  \mathrm{d}t. \]
By the exact same argument as Corollary~\ref{cor:LDP_Mallows_321}, we obtain the following.

\begin{cor}\label{cor:LDP_Mallows_231}
     Let $\beta\in \mathbb R.$  Let $\tau_n^\beta$ be a Mallows random permutation
     with parameter $q_n := \ee^{\beta/n}$, conditioned to avoid 231.
     Then the function $\varphi_{\tau_n^{\beta}}$ satisfies a large deviation principle on $\mathcal E$ at speed $n$ with rate function
     \[\varphi \mapsto A_\beta (\varphi) - \inf A_\beta.\]
\end{cor}

\section{An optimization problem for excursions}

We fix $\beta \in \R$.
\begin{prop}\label{prop:opti_problem_231}
	The action $A_{\beta}$ has a unique minimizer $\varphi_\beta$ in $\mathcal{E}$. 
    Moreover, for $\beta \leq 0$, we have $\varphi_{\beta}(t)=0$ for all $t \in [0,1]$. 
    For $\beta>0$, it is given by the formula
	\begin{equation}\label{eqn:defn_phi_beta}
		\varphi_\beta(t)= \frac{1}{\beta} \log\left( \frac{ 1 + \ee^{\beta} }{ 1 + \ee^{\beta(1-2t)}} \right) - t.
	\end{equation}
\end{prop}

\begin{proof}
	We first note that the case $\beta \leq 0$ is immediate since $H^{231}(\varphi) \geq 0$ for all $\varphi \in \mathcal{E}$, with equality if and only if $\varphi\equiv0$. 
    We now assume $\beta>0$.
	
	As in Part~\ref{part:321}, we will rely on Proposition~\ref{prop:max_concave}. We first note that $\mathcal{E}$ is compact for the topology of uniform convergence, and that the action $A_{\beta}$ is strictly convex as the sum of a strictly convex term and a linear one.
    We now define the subset
    \[
        \mathcal{E}' := \left\{ \varphi \in \mathcal{C}^1([0,1]) : 
        \begin{array}{ll} \varphi(0) = \varphi(1) = 0, \varphi(t) > 0 \text{ for all }0<t<1,\\
            \varphi'(0) > 0, \varphi'(1) < 0, \lVert \varphi' \rVert_\infty < 1 \end{array} 
        \right\} \,,
    \]
    and the set of bridges
    \[
        \mathcal{B} := \left\{ \alpha \in \mathcal{C}^1([0,1]) : \alpha(0) = \alpha(1) = 0 \right\} \,.
    \]
    It is straightforward to check that $\varphi_{\beta} \in \mathcal{E}'$ and that if $\varphi \in \mathcal{E}'$ then $V_\varphi(\mathcal{E}) = \mathcal{B},$ where we recall that $V_\varphi(K)$ is the set of directions $\alpha$ such that $\varphi + t\alpha \in K,$ for $t$ small enough.
    Therefore, by Proposition~\ref{prop:max_concave}, it is sufficient to check, for all $\alpha \in \mathcal{B}$,
    \begin{align*}
        0 &= \frac{\mathrm{d}}{\mathrm{d}s}\Big|_{s=0^+} A_{\beta} \left( \varphi+s\alpha \right)\\
        &= \int_0^1 \left( 2\alpha'(t) J'(\varphi_{\beta}'(t)) - 2\beta\alpha(t) \right) \mathrm dt\\
      	&= \int_0^1 \left( 2J'(\varphi'_{\beta}(t)) +2\beta t \right) \alpha'(t) \, \mathrm dt \, .
    \end{align*}
    Since $\int \alpha' = 0$, it is sufficient to check that $J'(\varphi'_{\beta}(t))+\beta t$ is constant, which is straightforward\footnote{
        Alternatively, as in Remark~\ref{rem:equadiff}, we could recover the formula for $\varphi_{\beta}$ by solving the differential equation $\log\left(\frac{1+\varphi'}{1-\varphi'}\right) = -2\beta t + c$.
    }.
\end{proof}

%

We can now prove the asymptotic estimate \eqref{eq:partition231} for the partition fonction of the model.

\begin{proof}[Proof of~\eqref{eq:partition231}]
The arguments are essentially the same as for equation~\eqref{eq:partition321}.
By Proposition~\ref{prop:LDP_231_pi_intro} and the first estimate in Lemma~\ref{lem:Varadhan}, we have that
\[ \lim_{n \rightarrow \infty} \frac{1}{n} \log \frac{Z_n^\beta}{Z_n^0} = -\inf A_\beta = - A_\beta(\varphi_\beta).\]

The latter can be (almost) explicitly computed. We have that
\[ \varphi_\beta'(t) = \frac{\ee^{\beta(1-2t)}-1}{\ee^{\beta(1-2t)}+1}\]
and using the change of variable $ y = \ee^{\beta(1-2t)},$ we get that
\[ H(\varphi_\beta) = 2 \int_0^1 J(\varphi_\beta'(t)) \dd t = 2 \log 2 +  \frac{1}{\beta} \int_{\ee^{-\beta}}^{\ee^{\beta}} \left( \frac{\log y }{1+y} - \frac{\log (1+y) }{y} \right) \dd y.\]
On the other hand,
\[ 2\beta \int_0^1 \varphi_\beta(t) \dd t = 2 \log (1+\ee^\beta) - \beta
- \frac{1}{\beta} \int_{\ee^{-\beta}}^{\ee^{\beta}}  \frac{\log (1+y) }{y} \dd y,\]
leading to the asymptotics anounced in equation~\eqref{eq:partition231}.
\end{proof}

We now conclude by proving the limit shape for 231-avoiding Mallows permutations.

\begin{proof}[Proof of Theorem~\ref{thm:phase_transition_231}]
	It follows immediately from Corollary~\ref{cor:LDP_Mallows_231} and Proposition~\ref{prop:opti_problem_231} that $\varphi_{\tau_n^{\beta}}$ converges almost surely (for any coupling) to $\varphi_{\beta}$. 
    By Lemma~\ref{lem:bicontinuity}, this implies the almost sure convergence of $\tau_n^{\beta}$ to the unique permuton $\mu_{\beta}^{231}$ such that $\varphi_{\mu_{\beta}^{231}}=\varphi_{\beta}$. 
    The only thing left to check is that $\mu_{\beta}^{231}$ corresponds to the description given by Theorem~\ref{thm:phase_transition_231}. 
    We first notice that for $\beta \leq 0$, we have $\varphi_{\beta}=0$, and the diagonal permuton $\mu_{\diagup}$ indeed satisfies $f_{\mu_{\diagup}}(t)=t$ for $0 \leq t \leq 1$, so $\varphi_{\mu_{\diagup}} \equiv 0$.
	
	We now assume $\beta>0$. Let $f_\beta\in\mathcal F$ be obtained from $\varphi_\beta\in\mathcal E$ by flipping, rotating and rescaling as described in Section~\ref{subsec:231_continuity}. The graph of $f_\beta$ is characterized by
	\[
	\left\{ (x,f_\beta(x)) : x\in[0,1] \right\}
	= \left\{ (t+\varphi_\beta(t), t-\varphi_\beta(t)) : t\in[0,1] \right\} \, .
	\]
	Plugging in the definition~\eqref{eqn:defn_phi_beta} of $\varphi_{\beta}$, we obtain
	\[ f_\beta(x)= 1 - \frac{1}{\beta} \log\left( 1 + \ee^\beta - \ee^{\beta x} \right) \, . \]
	
	From here, we deduce an explicit description of $\mu_\beta$ using Proposition~\ref{prop: support and density of 231 permuton}. Indeed, $\varphi_\beta(1-t) = \varphi_\beta(t)$ for all $t\in[0,1]$, and thus the graph of $f_\beta$ is symmetric with respect to the antidiagonal.
	The other assumptions of Proposition~\ref{prop: support and density of 231 permuton} are easy to check, and we have 
	\[
	x_\beta^* = \frac1\beta \log\left( \frac{1+\ee^\beta}{2} \right),\quad \quad
	f_\beta(x_\beta^*) = 1 - \frac1\beta \log\left( \frac{1+\ee^\beta}{2} \right) \,.
	\]
	Furthermore, for all $x \in [0,1]$, we have
	\[
		f_\beta'(x) = \frac{\ee^{\beta x}}{1 + \ee^\beta - \ee^{\beta x}} \,.
	\]
	Therefore, it follows from Proposition~\ref{prop: support and density of 231 permuton} that $\mu_\beta$ is the permuton with densities $1 \wedge \frac{\ee^{\beta x}}{1 + \ee^\beta - \ee^{\beta x}}$ on $\{ (x,f_\beta(x)) : x\in[0,1] \}$ and $0 \vee \frac{1 + \ee^\beta - 2\ee^{\beta x}}{1 + \ee^\beta - \ee^{\beta x}}$ on $\{ (x,1-x) : x\in[0,1] \}$. 
    Note that the first density (resp. the second) is equal to $1$ (resp. to $0$) if and only if $x\ge x_\beta^*$. This fits the formula given in Theorem~\ref{thm:phase_transition_231}.
\end{proof}


\appendix
\part{Appendix}
\section{Markov chain simulation}
\label{sec: simulations by reversible MC}

In this appendix, we explain how to generate the random permutations $\tau_n^{\beta, \alpha}$ for $\alpha\in\{321,231\}$, explaining how we got the simulations of \Cref{fig:simulations_Constraint_Mallows,fig:simulations_WithLimitCurves}.
This relies on interpreting the law of $\tau_n^{\beta, \alpha}$ as the stationary measure of an adequate Markov chain on the set $\Av_n(\alpha)$ of $\alpha$-avoiding permutations.
Specifically, we use a reversible Markov chain and the following standard lemma.

\begin{lem}
\label{lem: reversible and tilted MC}
    Let $X$ be a finite state space, $q\ge1$, and $f:X\to\R_+$.
    Write $\mu^f$ for the $(q,f)$-tilted measure on $X$, that is, $\mu^{q,f}(x) \propto q^{f(x)}$ for all $x\in X$.
    If $p : X^2 \to [0,1]$ is a transition kernel, denote by $\tilde p$ the transition kernel characterized by:
    \begin{align*}
        \forall (x,y)\in X^2,\quad
        \tilde p(x,y) = 
        \left\{
        \begin{array}{ll}
            p(x,y) & \text{if }x\ne y \text{ and }f(y)\ge f(x) \,;\\
            p(x,y)q^{f(y)-f(x)} & \text{if }x\ne y \text{ and }f(y)< f(x) \,;
        \end{array}
        \right.
    \end{align*}
    and $\tilde p(x,x)$ is adjusted so that this defines a transition kernel.
    If $p$ is reversible for the uniform measure on $X$, then $\tilde p$ is reversible for the measure $\mu^{q,f}$.
\end{lem}

In words, the tilted transition kernel $\tilde p$ is constructed as follows:
from $x\in X$, use $p$ to find a transition towards some $y\in X$.
If $f(y) \ge f(x)$, accept this transition.
If $f(y) < f(x)$, accept this transition with probability $q^{f(y)-f(x)}$; in case of reject, stay at $x$.

\begin{proof}
    This is a straightforward formal check.
    Let $(x,y)\in X$, and assume w.l.o.g.~that $x\ne y$ and $f(y)\le f(x)$.
    Therefore $\tilde p(x,y) = p(x,y)q^{f(y)-f(x)}$ and $\tilde p(y,x) = p(y,x)$, and thus $q^{f(x)}\tilde p(x,y) = q^{f(y)} p(x,y) = q^{f(y)} p(y,x) = q^{f(y)}\tilde p(y,x)$, where we used the reversibility of $p$ in the second equality.
\end{proof}

Finding a reversible Markov chain for $\tau_n^{\beta, \alpha}$ then boils down to finding a reversible Markov chain for the uniform measure on $\Av_n(\alpha)$, for which the inversion differential $\inv(\tau) - \inv(\sigma)$ can be efficiently computed on the transitions $\sigma \mapsto \tau$.

Recall that $\Av_n(\alpha)$ is in one-to-one correspondence with the set $\Dyck_n$ of Dyck paths of length $2n$, both for $\alpha=321$ and $\alpha=231$.
If $d : [0,2n] \to \R_+$ is a Dyck path and $i\in [1,2n-1]$, transform $d$ into $d^{(i)} \in \Dyck_n$ as follows:
\begin{itemize}
    \item If $i$ is a peak index for $d$, i.e.~if $d(i-1) = d(i)-1 = d(i+1)$, and if $d(i)\ge2$, then transform this peak into a valley in $d^{(i)}$.
    That is, let $d^{(i)}(j) := d(j)$ for all indices $j\ne i$, and let $d^{(i)}(i) := d(i)-2$.
    \item If $i$ is a valley index for $d$, i.e.~if $d(i-1) = d(i)+1 = d(i+1)$, then transform this valley into a peak in $d^{(i)}$.
    That is, let $d^{(i)}(j) := d(j)$ for all indices $j\ne i$, and let $d^{(i)}(i) := d(i)+2$.
    \item In all other cases, let $d^{(i)} := d$.
\end{itemize}
Then define a transition kernel on $\Dyck_n$ via $p(d,\cdot) = \frac{1}{2n-1} \sum_{i=1}^{2n-1} \delta_{d^{(i)}}$, that is, apply the above transform at a uniformly random index.
It is straightforward to check that this defines an irreducible Markov chain on $\Dyck_n$ that is reversible for the uniform measure.

Hence we can approximate the uniform measure on $\Av_n(\alpha)$ by running this Markov chain on $\Dyck_n$ for ``a long time'', and then applying the one-to-one correspondence from $\Dyck_n$ to $\Av_n(\alpha)$.
To approximate $\tau_n^{\beta,\alpha}$ instead, it suffices to use the transition kernel $\tilde p$ given by Lemma~\ref{lem: reversible and tilted MC}.
In the case $\alpha=321$, $\inv(d^{(i)}) - \inv(d)$ can be computed in constant time thanks to Lemma~\ref{lem:inv_sigma}.
This is even simpler for $\alpha=231$ thanks to Proposition~\ref{prop: inv of 231 avoiding permutation}: $\inv(d^{(i)}) - \inv(d)$ is equal to $+1$ if a valley was transformed into a peak, and to $-1$ if a peak was transformed into a valley.

Since we do not know the mixing time of the above Markov chains\footnote{
   It is known that the Markov chain $p$ on $\Dyck_n$ has mixing time $O(n^3 \log(n))$, see e.g.~\cite{mcshine1997mixing,wilson2004mixing}, but this does not give us an estimate for the mixing time of $\tilde p$.
}, we use the following standard heuristics.
We start from two very different initial conditions (e.g., the Dyck paths with minimal and maximal global heights respectively; the first one corresponds to the permutations $\mathrm{id}$, whereas the second one correpsonds to $2\,3\dots(n-1)1$ for $\alpha=321$ and to $n\,(n-1)\dots1$ for $\alpha=231$), 
perform the Markov chain steps on both (using the same random index $i$ on both at each step),
and compare the resulting pictures.
If the two pictures look similar, it means that, heuristically, the number of steps was large enough so that the chain forgets its initial condition.
In practice, we get satisfying results after 100 million steps;
to be on the safe side, we
used 300 million steps to generate the pictures of \Cref{fig:simulations_Constraint_Mallows,fig:simulations_WithLimitCurves}.

\section{Useful lemmas}\label{sec:appendix}

\begin{lem}\label{lem:cv_push_forward}
    Let $\mu_n$ be a sequence of finite measures on a space $\Omega$ converging in total variation to a finite measure $\mu$,
    and $f_n$ a sequence of measurable functions on $\Omega$ converging $\mu$-a.e.~to a function $f$.
    Then $(f_n)_{\#}\mu_n$ converges weakly to $f_{\#}\mu$.
\end{lem}
\begin{proof}
    The distance $\dTV( (f_n)_{\#}\mu_n, (f_n)_{\#}\mu)$ is at most $\dTV(\mu_n,\mu)$, which tends to $0$.
    Hence it is enough to prove that $(f_n)_{\#}\mu$ converges weakly to $f_{\#}\mu$.
    Letting $X$ be a random variable with law $\mu/\mu(\Omega)$, it holds by assumption that $f_n(X)$ converges almost surely to $f(X)$.
    Since almost sure convergence implies convergence in distribution, the law $(f_n)_{\#}\mu/\mu(\Omega)$ of $f_n(X)$ converges weakly to that of $f(X)$, which is $f_{\#}\mu/\mu(\Omega)$.
\end{proof}

\begin{lem}\label{lem:inf}
Let us consider two sets $A$ and $B$ and a function $h$ on $A \cup B$.
It holds that
\begin{equation}\label{eq:infAB}
\inf_{z \in A \cup B} h(z) +\inf_{z \in A \cap B} h(z)
\ge \inf_{z \in A} h(z) + \inf_{z \in B} h(z).
\end{equation}
Moreover, if the inequality is strict, then
\begin{equation}\label{eq:AsansB}
    \inf_{z \in A} h(z) = \inf_{z \in A \backslash B} h(z) < \inf_{z \in A \cap B} h(z),
\end{equation}
and the same holds if we switch the roles of $A$ and $B$.
\end{lem}
\begin{proof}
We can assume without loss of generality
that $\inf_{z \in A \cup B} h(z) =\inf_{z \in B} h(z)$.
Then~\eqref{eq:infAB} follows from the trivial inequality
$\inf_{z \in A \cap B} h(z) \ge \inf_{z \in A} h(z)$. This inequality being strict implies 
\[ \inf_{z \in A \cap B} h(z) > \inf_{z \in A} h(z) \geq \inf_{z \in A \cup B} h(z) = \inf_{z \in B} h(z). \]
The first inequality implies~\eqref{eq:AsansB}, and the strict inequality between the first and last terms implies~\eqref{eq:AsansB} with $A$ and $B$ switched.
\end{proof}

\textit{Optimization on convex spaces.}
Let $E$ be a topological vector space and $K \subset E$ convex.
We say that $f:K \to \R$ is strictly convex if for any $x \in K$ and $y \in E$, the function $t \to f(x+ty)$ is strictly convex on the interval where it is defined
(with the convention that a function defined on a singleton of $\R$ is strictly convex).
For any $x \in K$, we denote by $V_x(K)$ the set of directions $y \in E$ such that $x+ty \in K$ for $t>0$ small enough.

\begin{prop}\label{prop:max_concave}
    Let $E$ be a topological vector space and $K \subset E$ compact and convex, and let $f:K \to \mathbb{R}$ be lower semicontinuous and strictly convex.
    Then $f$ has a minimum which is attained at exactly one point $x_0 \in K$.
    Moreover, if $x \in K$ is such that for any $y \in V_x(K)$, we have
    \[ \frac{\mathrm{d}}{\mathrm{d}t}\Big|_{t=0^+} f(x+ty)=0 ,\]
    then $x=x_0$.
\end{prop}

\begin{proof}
    The existence of the minimum is immediate by lower semicontinuity.
    Moreover, if $f(x_1)=f(x_2)$ with $x_1 \ne x_2$, then the strict convexity of $f$ implies $f \left( \frac{x_1+x_2}{2}\right) <f(x_1)$, which gives the uniqueness of the minimum.
    Finally, if $x$ satisfies the given condition and $x' \in K \backslash\{x\}$, then let $y=x'-x \in V_x(K)$. The function $h:t \mapsto f(x+ty)$ has right-derivative $0$ at $t=0$ and is strictly convex, so $f(x')=h(1)>h(0)=f(x)$.
    We conclude that $x$ must be the unique minimum $x_0$ of $f$ on $K$, as claimed.
\end{proof}

\section*{Acknowledgements}
This work was initiated during an open problem workshop organized
as part of the ANR project LOUCCCOUM (Large Objects Under Combinatorial Constraints and Outside Uniform Models, ANR-24-CE40-7809). 
All authors are partially supported by this grant.
\bibliographystyle{bibli_perso}
\bibliography{BiblioANR}


\end{document}